\documentclass[reqno]{amsart}

\pdfoutput=1

\usepackage[alphabetic]{amsrefs}
\usepackage{amsmath}
\usepackage{amsfonts}
\usepackage{amssymb}
\usepackage{enumerate}
\usepackage{amstext}
\usepackage{amsbsy}
\usepackage{amsopn}
\usepackage{bbm,amsthm}
\usepackage{amscd}
\usepackage[pdftex]{color}
\usepackage{amsxtra}
\usepackage{upref}
\usepackage{epstopdf}
\usepackage{graphicx,color}
\usepackage{hyperref}
\usepackage{longtable}
\usepackage{graphicx}
\usepackage[francais, english]{babel}

\DeclareFontFamily{OML}{rsfs}{\skewchar\font'177}
\DeclareFontShape{OML}{rsfs}{m}{n}{ <5> <6> rsfs5 <7> <8> <9> rsfs7
  <10> <10.95> <12> <14.4> <17.28> <20.74> <24.88> rsfs10 }{}
\DeclareMathAlphabet{\mathfs}{OML}{rsfs}{m}{n}

\newtheorem{theorem}{Theorem}
\newtheorem{lemma}[theorem]{Lemma}
\newtheorem{proposition}[theorem]{Proposition}

\theoremstyle{definition}

\theoremstyle{remark}
\newtheorem{remark}[theorem]{\bf Remark}

\numberwithin{equation}{section}
\numberwithin{theorem}{section}

\newcommand{\intav}[1]{\mathchoice {\mathop{\vrule width 6pt height 3 pt depth  -2.5pt
\kern -8pt \intop}\nolimits_{\kern -6pt#1}} {\mathop{\vrule width
5pt height 3  pt depth -2.6pt \kern -6pt \intop}\nolimits_{#1}}
{\mathop{\vrule width 5pt height 3 pt depth -2.6pt \kern -6pt
\intop}\nolimits_{#1}} {\mathop{\vrule width 5pt height 3 pt depth
-2.6pt \kern -6pt \intop}\nolimits_{#1}}}

\newcommand{\intavl}[1]{\mathchoice {\mathop{\vrule width 6pt height 3 pt depth  -2.5pt
\kern -8pt \intop}\limits_{\kern -6pt#1}} {\mathop{\vrule width 5pt
height 3  pt depth -2.6pt \kern -6pt \intop}\nolimits_{#1}}
{\mathop{\vrule width 5pt height 3 pt depth -2.6pt \kern -6pt
\intop}\nolimits_{#1}} {\mathop{\vrule width 5pt height 3 pt depth
-2.6pt \kern -6pt \intop}\nolimits_{#1}}}

\newcommand{\un}{\underline}

\newcommand{\ve}{\varepsilon}

\newcommand{\wh}{\widehat}
\newcommand{\vt}{\vartheta}
\newcommand{\vf}{\varphi}

\newcommand{\R}{\mathbb{R}}
\newcommand{\N}{\mathbb{N}}

\newcommand{\Z}{\mathbb{Z}}

\renewcommand{\exp}[1]{{\rm exp}_{#1}}

\newcommand{\Hol}[1]{{\rm Hol}_{#1}}

\begin{document}


\title[Symbolic dynamics for 1--dim maps with nonuniform expansion]{Symbolic dynamics for one dimensional \\
maps with nonuniform expansion}

\author{Yuri Lima}
\thanks{The author is supported by Instituto Serrapilheira, grant Serra-1709-20498.}
\date{\today}

\address{Yuri Lima, Departamento de Matem\'atica, Universidade Federal do Cear\'a (UFC), Campus do Pici,
Bloco 914, CEP 60440-900. Fortaleza -- CE, Brasil}
\email{yurilima@gmail.com}

\begin{abstract}
Given a piecewise $C^{1+\beta}$ map of the interval, possibly with \mbox{critical} points and
discontinuities, we construct a symbolic model for invariant probability measures with nonuniform expansion
that do not approach the critical points and discontinuities exponentially fast almost surely.
More specifically, for each $\chi>0$ we construct a finite-to-one H\"older continuous map from
a countable topological Markov shift to the natural extension of the interval map, that codes the
lifts of all invariant probability measures as above with Lyapunov exponent greater than $\chi$ almost everywhere.
\end{abstract}

\maketitle


\section{Introduction}\label{Section-introduction}

The quadratic family $\{f_a\}_{0\leq a\leq 4}$ is the family of one dimensional interval
maps $f_a:[0,1]\to[0,1]$, $f_a(x)=ax(1-x)$. Although simple to describe, it exhibits
complicated dynamical behavior: for a set of parameters of positive Lebesgue measure, $f_a$
has an absolutely continuous invariant measure with positive Lyapunov exponents
\cite{Jakobson,BenedicksCarleson85,BenedicksCarleson91},
see also \cite{Yoccoz}.
The idea to prove this is to construct a partition of the interval with good symbolic properties that
allows to understand the orbit of the critical point $x=0.5$, so that for many parameters the
critical value $f_a(0.5)$ has positive Lyapunov exponent (this latter property is known as the
{\em Collet-Eckmann condition}).
This idea has far reaching applications, see e.g. \cite{Mane-CMP,Graczyk-Swiatek,Lyubich-denseness}.

\medskip
The present works goes in the reverse direction of the above idea: it considers piecewise $C^{1+\beta}$ maps $f:[0,1]\to[0,1]$
of the interval with positive Lyapunov exponent and constructs {\em finite-to-one H\"older continuous} symbolic
extensions of the maps. We require $f$ to satisfy the regularity conditions (A1)--(A3), that will be shortly described.
These conditions allow $f$ to have both critical points (where the first derivate vanishes) and
discontinuities (where the first derivative can explode), and include the quadratic family,
multimodal maps with non-flat critical points, piecewise continuous maps with discontinuities of polynomial type,
and combinations of these two classes, see figure \ref{figure-1dim} for examples.  

\begin{figure}[hbt!]\label{figure-1dim}
\centering
\def\svgwidth{9cm}
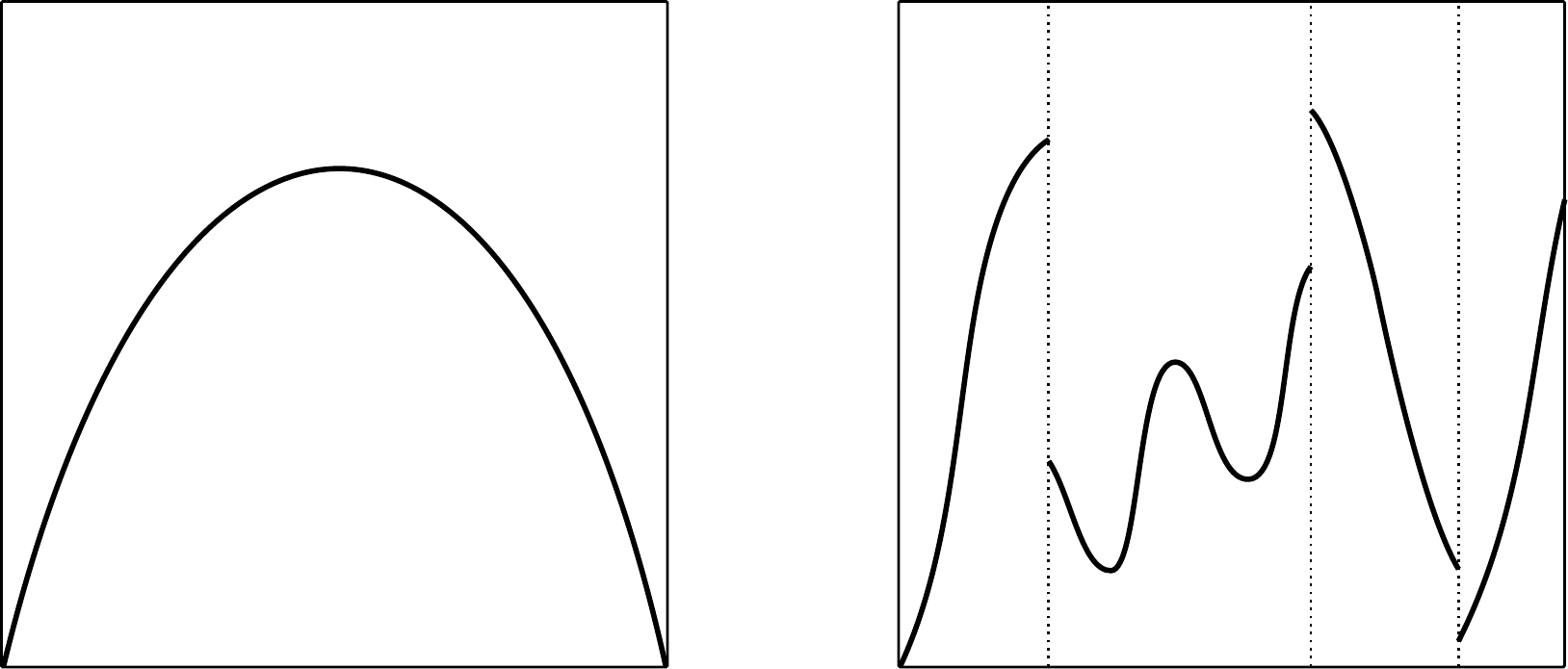
\caption{Examples of maps covered by our results.}
\end{figure}

\medskip
Our main result is the construction of a symbolic model for the natural extension
$\wh f:\wh{[0,1]}\to\wh{[0,1]}$ of $f$.
See Section \ref{Section-natural-extension} for the definition of the natural extension and its main properties.
Let us describe which maps we consider and which measures we are able to code.
For the sake of simplicity, we consider maps defined on the interval $M=[0,0.5]$, since this interval
has diameter less than one (and so we do not need to introduce multiplicative constants in the assumptions
(A1)--(A3) below).
Let $f:M\to M$ be a map with {\em discontinuity set} $\mathfs D$. We assume that $f$ is $C^{1+\beta}$
in the set $M\backslash\mathfs D$, for some $\beta\in(0,1)$. Let $\mathfs C:=\{x\in M\backslash\mathfs D:df_x=0\}$
denote the {\em critical set} of $f$.

\medskip
\noindent
{\sc Singular set}: The {\em singular set} of $f$ is $\mathfs S:=\mathfs C\cup\mathfs D$.

\medskip
We allow $\mathfs S$ to be infinite, e.g. when $f$ is the Gauss map.
Let $B(x,r)\subset M$ denote the ball with center $x$ and radius $r$.
We assume that $f$ satisfies the following properties.

\medskip
\noindent
{\sc Regularity of $f$:} There exist constants $a,\mathfrak K>1$ s.t. for all $x\in M$ with 
$x,f(x)\notin\mathfs S$ there is $\min\{d(x,\mathfs S)^a,d(f(x),\mathfs S)^a\}<\mathfrak r(x)<1$
s.t. for $D_x:=B(x,2\mathfrak r(x))$ and $E_x:=B(f(x),2\mathfrak r(x))$ the following holds:
\begin{enumerate}[.......]
\item[(A1)] The restriction of $f$ to $D_x$ is a diffeomorphism onto its image;
the inverse branch of $f$ taking $f(x)$ to $x$ is a well-defined diffeomorphism from 
$E_x$ onto its image.
%
\item[(A2)] For all $y\in D_x$ it holds $d(x,\mathfs S)^a\leq |df_y|\leq d(x,\mathfs S)^{-a}$; for
all $z\in E_x$ it holds $d(x,\mathfs S)^a\leq |dg_z|\leq d(x,\mathfs S)^{-a}$,
where $g$ is the inverse branch of $f$ taking $f(x)$ to $x$.
\item[(A3)] For all $y,z\in D_x$ it holds $|df_y-df_z|\leq\mathfrak K|y-z|^\beta$;
for all $y,z\in E_x$ it holds $|dg_y-dg_z|\leq\mathfrak K|y-z|^\beta$.
\end{enumerate}

\medskip
Now we describe the measures that we code. We borrow
the notation from \cite{Lima-Matheus}. Let $\mu$ be an $f$--invariant probability measure.

\medskip
\noindent
{\sc $f$--adapted measure:} The measure $\mu$ is called {\em $f$--adapted} if $\log d(x,\mathfs S)\in L^1(\mu)$.
A fortiori, $\mu(\mathfs S)=0$.

\medskip
\noindent
{\sc $\chi$--expanding measure:} Given $\chi>0$, the measure $\mu$
is called {\em $\chi$--expanding} if $\lim_{n\to\infty}\tfrac{1}{n}\log |df_x^n|>\chi$ for 
$\mu$--a.e. $x\in M$.

\medskip
The next theorem is the main result of this paper. Below, $\widehat{f}:\widehat{M}\to\widehat{M}$
is the natural extension of $f$, and $\widehat\mu$ is the lift of $\mu$, see Subsection
\ref{Section-natural-extension} for the definition. 

\begin{theorem}\label{Thm-Main}
Assume that $f$ satisfies {\rm (A1)--(A3)}. For all $\chi>0$, there exists a countable topological
Markov shift $(\Sigma,\sigma)$ and $\pi:\Sigma\to\wh M$ H\"older continuous s.t.:
\begin{enumerate}[{\rm (1)}]
\item $\pi\circ\sigma=\wh f\circ\pi$.
\item $\pi[\Sigma^\#]$ has full $\wh\mu$--measure for every $f$--adapted $\chi$--expanding measure $\mu$.
\item For all $\wh{x}\in \pi[\Sigma^\#]$, the set $\{\un v\in\Sigma^\#:\pi(\un v)=\wh x\}$ is finite.
\end{enumerate}
\end{theorem}

Above, $\Sigma^\#$ is the {\em recurrent set} of $\Sigma$; it carries all $\sigma$--invariant
probability measures, see Section \ref{Section-preliminaries}.
Therefore we are able to code simultaneously all the measures with nonuniform expansion
greater than $\chi$ almost everywhere that do not
approach the singular set exponentially fast. 

\medskip
It is important to make some comments on the assumption of $f$--adaptedness.
By the Birkhoff ergodic theorem, if $\mu$ is $f$--adapted then $\lim_{n\to\infty}\tfrac{1}{n}\log d(f^n(x),\mathfs S)=0$
$\mu$--a.e. Ledrappier considered this latter condition for interval maps with critical points \cite{Ledrappier-acip},
where he used the terminology {\em non-degenerate measure}. Katok and Strelcyn implicitly used
that the Lebesgue measure is adapted to billiard maps and then used that 
$\lim_{n\to\infty}\tfrac{1}{n}\log d(f^n(x),\mathfs S)=0$ a.e. \cite{Katok-Strelcyn}.
For an invariant measure of a three dimensional flow with positive speed,
Lima and Sarig constructed a Poincar\'e section for which the respective Poincar\'e return map $f$
satisfies $\lim_{n\to\infty}\tfrac{1}{n}\log d(f^n(x),\mathfs S)=0$
almost surely with respect to the induced measure \cite{Lima-Sarig}, and Lima and Matheus
used the assumption of adaptability in their coding of billiard maps \cite{Lima-Matheus}. 
For one dimensional maps satisfying (A1)--(A3), if $\mathfs D=\emptyset$
and $\mathfs C$ is finite then every ergodic invariant measure that is not 
supported in an attracting periodic point satisfies $\lim_{n\to\infty}\tfrac{1}{n}\log d(f^n(x),\mathfs S)=0$
a.e. \cite{Przytycki-Lyapunov},
see also \cite[Appendix A]{Rivera-Letelier} for a proof that works under weaker assumptions.
It would be interesting to obtain the
same conclusion when $\mathfs S$ is finite.

\medskip
Now we make a comparison of our method with the one developed
by Hofbauer \cite{Hofbauer-intrinsic-I,Hofbauer-intrinsic-II}, known as {\em Hofbauer towers}.
These towers were first constructed to analyze measures of maximal entropy,
and provide a precise combinatorial description of one dimensional maps.
In comparison to Hofbauer's method, our method explores the nonuniform expansion of $\chi$--expanding
measures.
It constitutes the first implementation, for non-invertible systems,
of the recent constructions of Markovian symbolic dynamics for nonuniformly
hyperbolic systems \cite{Sarig-JAMS,Lima-Sarig,Lima-Matheus,Ben-Ovadia-2019}.
The novelty of the present paper is that,
contrary to Hofbauer's method, our construction has the following advantages:
\begin{enumerate}[$\circ$]
\item The extension map $\pi$ constructed in Theorem \ref{Thm-Main}
is {\em H\"older continuous}.
\item While Hofbauer's method only works in very specific higher dimensional cases
(see Section \ref{Section-related-literature} below), our method is more robust
and will be extended, in a forthcoming work, to higher dimensional maps such as
complex-valued functions, Viana maps, and general nonuniformly hyperbolic maps.
\end{enumerate}
In addition to these, we emphasize that $\mathfs S$ can be infinite.

\subsection{Applications}

Assume that $f$ satisfies (A1)--(A3) and has finite and positive topological entropy $h\in (0,+\infty)$. 
In this section, we discuss two applications of Theorem \ref{Thm-Main}:
\begin{enumerate}[$\circ$]
\item Estimates on the number of periodic points.
\item Understanding of equilibrium measures.
\end{enumerate}
The first can be obtained in great generality, but the second is subtler
due to the presence of discontinuities and eventual non-finiteness of $\mathfs S$.
As a matter of fact, the measure of maximal
entropy may not exist, even for $C^r$ maps \cite{Buzzi-smooth-interval-maps,Ruette-mme-interval}.
When $\mathfs D=\emptyset$ and $\mathfs C$ is finite, $f$ is called
a multimodal map. There are many results on ergodic and statistical properties
of equilibrium measures for one dimensional maps in various contexts, such as uniformly expanding maps
\cite{Lasota-Yorke,Smorodisnky-Bernoulli,Takahashi-isomorphisms,Bowen-interval},
piecewise continuous maps \cite{Hofbauer-intrinsic-I,Hofbauer-intrinsic-II},
piecewise monotone maps \cite{Hofbauer-Keller}, and multimodal maps
\cite{Ledrappier-acip,Denker-Keller-Urbanski,Bruin-Keller,Pesin-Senti,Bruin-Todd-1,Bruin-Todd-2,Iommi-Todd}.

\medskip
As mentioned above, Theorem \ref{Thm-Main} allows $\mathfs S$ to be infinite.
When there is an $f$--adapted measure of maximal entropy, we obtain
exponential estimates on the number of periodic points.

\begin{theorem}\label{Thm-Periodic-Points}
Assume $f$ satisfies {\rm (A1)--(A3)} with topological entropy $h\in(0,\infty)$.
If there is an $f$--adapted measure of maximal entropy, then 
$\exists C>0,p\geq 1$ s.t. $\#\{x\in[0,1]:f^{pn}(x)=x\}\geq C e^{hpn}$
for all $n\geq 1$.
\end{theorem}

\begin{proof}
The set of periodic points of $f$ is in bijection with the set of periodic
points of its natural extension $\wh f$. Since $(\Sigma,\sigma)$ is a finite extension of $\wh f$,
the growth rate of periodic points of $\wh f$ is at least that of $\sigma$.
Let $\mu$ be an ergodic $f$--adapted measure of maximal entropy for $f$.
Its lift $\wh\mu$ is an ergodic measure of maximal entropy for $\wh f$ satisfying the assumptions of 
Theorem \ref{Thm-Main}. Proceeding as in \cite[\S 13]{Sarig-JAMS}, $\wh\mu$ lifts
to an ergodic measure of maximal entropy $\nu$ for $(\Sigma,\sigma)$.
By \cite{Gurevich-Topological-Entropy,Gurevich-Measures-Of-Maximal-Entropy},
$\nu$ is supported on a topologically transitive countable topological Markov shift,
and there is $p\geq 1$ s.t. for every vertex $v$ it holds
$\#\{\un v\in\Sigma: \sigma^{pn}(\un v)=\un v,v_0=v\}\asymp e^{hpn}$.
Hence, there exists $C>0$ s.t. $\#\{x\in[0,1]:f^{pn}(x)=x\}\geq C e^{hpn}$.
\end{proof}

The number $p$ equals the period of the transitive component supporting
the measure of maximal entropy. For surface maps, Buzzi gave conditions for which 
$p=1$ \cite{Buzzi-2019}, and we expect his method also applies here.

\medskip
Now we turn attention to equilibrium measures. 
Let $(X,\mathfs A,T)$ be a probability measure-preserving system, and $\vf:X\to\R$ a continuous potential.
The following definitions are standard.

\medskip
\noindent
{\sc Topological pressure:} The {\em topological pressure} of $\vf$ is
$P_{\rm top}(\vf):=\sup\{h_\mu(T)+\int\vf d\mu\}$,
where the supremum ranges over all $T$--invariant probability measures.

\medskip
\noindent
{\sc Equilibrium measure:} An {\em equilibrium measure} for $\vf$ is a $T$--invariant probability 
measure $\mu$ s.t. $P_{\rm top}(\vf)=h_\mu(T)+\int\vf d\mu$.

\medskip
Due to the presence of discontinuities, the topological pressure 
might be infinite (e.g. the Gauss map has infinite topological entropy),
and equilibrium measures may not exist.
Theorem \ref{Thm-Main} provides countability results for a class
of potentials and equilibrium measures:
we require $\vf$ to be a bounded H\"older continuous potential with
finite topological pressure, and consider equilibrium measures that are $f$--adapted
and $\chi$--expanding for some $\chi>0$. Call
a measure {\em expanding} if it is $\chi$--expanding for some $\chi>0$.
Observe that, when the Ruelle inequality applies,
every ergodic measure with positive metric entropy is expanding.

\begin{theorem}\label{Thm-equilibrium-states}
Assume that $f$ satisfies {\rm (A1)--(A3)}. Every equilibrium measure of a 
bounded H\"older continuous potential with finite topological pressure
has at most countably many $f$--adapted expanding
ergodic components. Furthermore, the lift to $\wh M$ of each such ergodic
component is Bernoulli up to a period.
\end{theorem}

\begin{proof}
Let $\vf:M\to\R$ be bounded, H\"older continuous, with $P_{\rm top}(\vf)<\infty$.
We prove that, for each $\chi>0$, $\vf$ possesses at most countably
many $f$--adapted $\chi$--expanding equilibrium measures. The first part of the theorem
follows by taking the union of these measures for $\chi_n=\tfrac{1}{n}$, $n>0$. 

\medskip
Fix $\chi>0$. Let $\vt:\wh M\to M$ be the projection into the zeroth coordinate,
see Subsection \ref{Section-natural-extension}, and
define $\wh\vf:\wh M\to \R$ by $\wh\vf=\vf\circ\vt$.
Then $\wh\vf$ is bounded, H\"older continuous and has finite topological pressure
$P_{\rm top}(\wh\vf)=P_{\rm top}(\vf)$. Furthermore, $\mu$ is an ergodic equilibrium measure
for $\vf$ iff $\wh\mu$ is an ergodic equilibrium measure for $\wh\vf$. When $\mu$ is additionally
$f$--adapted and $\chi$--expanding, we can apply the procedure in  \cite[\S 13]{Sarig-JAMS}
and Theorem \ref{Thm-Main} to lift $\wh\mu$ to an ergodic equilibrium measure $\nu$
in $(\Sigma,\sigma)$. By ergodicity, $\nu$ is carried by a topologically transitive countable
topological Markov shift.
The potential associated to $\nu$ is $\Phi=\wh\vf\circ\pi$, which is bounded, H\"older continuous
and has finite topological pressure $P_{\rm top}(\Phi)=P_{\rm top}(\wh\vf)=P_{\rm top}(\vf)$. 
By \cite[Thm. 1.1]{Buzzi-Sarig}, each topologically transitive countable topological Markov shift
carries at most one equilibrium
measure for $\Phi$, hence there are at most countably many such $\nu$. This proves the
first part of the theorem. By \cite{Sarig-Bernoulli-JMD},
each such $\nu$ is Bernoulli up to a period. Since this latter property is preserved by finite-to-one
extensions, the same occurs to $\wh\mu$. This concludes the proof of the theorem.
\end{proof}


\subsection{Related literature}\label{Section-related-literature}

Symbolic models in dynamics have a longstanding history that can be traced back to the work
of Hadamard on closed geodesics of hyperbolic surfaces, see e.g. \cite{Katok-Ugarcovici-Symbolic}.
The late sixties and early seventies saw a great deal of development
of symbolic dynamics for uniformly hyperbolic diffeomorphisms and flows, through the works
of Adler \& Weiss \cite{Adler-Weiss-PNAS,Adler-Weiss-Similarity-Toral-Automorphisms},
Sina{\u\i} \cite{Sinai-Construction-of-MP,Sinai-MP-U-diffeomorphisms},
Bowen \cite{Bowen-MP-Axiom-A,Bowen-Symbolic-Flows},
Ratner \cite{Ratner-MP-three-dimensions,Ratner-MP-n-dimensions}.
Below we discuss other relevant contexts.

\medskip
\noindent
{\sc Hofbauer towers:} Takahashi developed a combinatorial method
to construct an isomorphism between a large subset $X$ of the natural extension of $\beta$--shifts
and countable topological Markov shifts \cite{Takahashi-isomorphisms}.
Hofbauer proved that $X$ carries all measures of positive entropy and hence
$\beta$--shifts have a unique measure of maximal entropy  \cite{Hofbauer-beta-shifts}.
Hofbauer later extended his construction to piecewise continuous interval maps
\cite{Hofbauer-intrinsic-I,Hofbauer-intrinsic-II}. The symbolic models obtained by his methods
are called {\em Hofbauer towers}, and they have been extensively used to establish ergodic properties
of one dimensional maps.

\medskip
\noindent
{\sc Higher dimensional Hofbauer towers:} Buzzi constructed Hofbauer towers for piecewise
expanding affine maps in any dimension \cite{Buzzi-affine-maps}, for perturbations of fibered
products of one dimensional maps \cite{Buzzi-produits-fibres}, and for arbitrary piecewise invertible
maps whose entropy generated by the boundary of some dynamically relevant partition is {\em less} than
the topological entropy of the map \cite{Buzzi-multidimensional}. 
These Hofbauer towers carry all invariant measures with entropy close enough to the topological
entropy of the system. We remark that, contrary to us, Buzzi's conditions make no reference to the 
nonuniform hyperbolicity of the system.

\medskip
\noindent
{\sc Inducing schemes:} Many systems, although not hyperbolic, do have sets on which
it is possible to define a (not necessarily first) return map on which the map becomes uniformly hyperbolic.
This process is known as {\em inducing}. Indeed, Hofbauer towers can be seen as inducing schemes
for which the map becomes uniformly expanding, see \cite{Bruin-inducing-Hofbauer} for this relation.
It is possible to understand it ergodic theoretical properties of invariant measure that lifts to an inducing scheme,
as done for one-dimensional maps \cite{Pesin-Senti}, higher dimensional ones that do not have
full ``boundary entropy'' \cite{Pesin-Senti-Zhang}, and expanding measures \cite{Pinheiro-expanding}.

\medskip
\noindent
{\sc Yoccoz puzzles:} Yoccoz constructed Markov structures for quadratic maps of the complex
plane, nowadays called {\em Yoccoz puzzles}, and used them to establish the MLC conjecture
for finitely renormalizable parameters and also a proof of Jakobson's theorem, see \cite{Yoccoz}. 

\medskip
\noindent
{\sc Nonuniformly hyperbolic diffeomorphisms:} Katok constructed horseshoes
of large topological entropy for $C^{1+\beta}$ diffeomorphisms \cite{KatokIHES}.
These horseshoes usually have zero measure for measures of maximal entropy.
Sarig constructed a ``horseshoe'' of full entropy for $C^{1+\beta}$ surface diffeomorphisms \cite{Sarig-JAMS}:
for each $\chi>0$ there is a countable topological Markov shift that is an
extension of the diffeomorphism and codes all $\chi$--hyperbolic measures simultaneously.
Ben-Ovadia extended the work of Sarig to higher dimension \cite{Ben-Ovadia-2019}.

\medskip
\noindent
{\sc Nonuniformly hyperbolic three-dimensional flows:} Lima and Sarig
constructed symbolic models for nonuniformly hyperbolic three dimensional flows with positive
speed \cite{Lima-Sarig}. The idea is to build a ``good'' Poincar\'e section and
construct a Markov partition for the Poincar\'e return map $f$.

\medskip
\noindent
{\sc Billiards:} Dynamical billiards are maps with discontinuities.
Katok and Strelcyn constructed invariant manifolds for nonuniformly hyperbolic
billiard maps \cite{Katok-Strelcyn}.
Bunimovich, Chernov, and Sina{\u\i} constructed countable Markov partitions for two dimensional
dispersing billiard maps \cite{Bunimovich-Chernov-Sinai}.
Young construced an inducing scheme for certain two dimensional dispersing billiard maps and
used it to prove exponential decay of correlations \cite{Young-towers}. 
All these results are for Liouville measures. Lima and Matheus constructed countable
Markov partitions for two dimensional billiard maps and nonuniformly hyperbolic (not necessarily Liouville)
measures that are adapted to the billiard map \cite{Lima-Matheus}.


\subsection{Method of proof}

The proof of Theorem \ref{Thm-Main} is based on \cite{Sarig-JAMS}, \cite{Lima-Sarig} and \cite{Lima-Matheus},
and follows the steps below:
\begin{enumerate}[(1)]
\item The derivative cocycle $df$ induces an invertible cocycle $\wh{df}$, defined on a fiber bundle over
the natural extension space $\wh M$, with the same spectrum as $df$.
\item If $\mu$ is $f$--adapted and $\chi$--expanding, then $\wh\mu$--a.e. $\wh x\in\wh M$
has a Pesin chart $\Psi_{\wh x}:[-Q_\ve(\wh x),Q_\ve(\wh x)]\to M$ s.t.
$\lim_{n\to\infty}\tfrac{1}{n}\log Q_\ve(\wh f^n(\wh x))=0$.
In Pesin charts, the inverse branches of $f$ are uniform contractions.
\item Introduce the $\ve$--chart $\Psi_{\wh x}^p$ as the restriction of $\Psi_{\wh x}$
to $[-p,p]$. The parameter $p$ gives a definite size for the unstable manifold at $\wh x$, hence different
$\ve$--charts $\Psi_{\wh x}^p,\Psi_{\wh x}^q$ define unstable manifolds of different sizes.
\item Construct a countable collection of $\ve$--charts that are dense
in the space of all $\ve$--charts, where denseness is defined in terms of
finitely many parameters of $\wh x$.
\item Draw an edge $\Psi_{\wh x}^p\leftarrow\Psi_{\wh y}^q$ between $\ve$--charts 
when an inverse branch of $f$ can be represented in these charts by a uniform contraction
and the parameter $q$ is as large as possible. Each path of $\ve$--charts defines 
an element of $\wh M$, and this coding induces a countable cover on $\wh M$. The requirement 
that $q$ is as large as possible guarantees that this cover is locally finite.
\item Apply a refinement procedure to this cover. The resulting partition defines
a countable topological Markov shift $(\Sigma,\sigma)$ and a coding $\pi:\Sigma\to \wh M$
that satisfy Theorem \ref{Thm-Main}.
\end{enumerate}

\medskip
It is important to stress the importance of having a countable locally finite cover in step (5). 
If not, its refinement could be an uncountable partition (imagine e.g. the cover of $\R$ 
by intervals with rational endpoints). When the countable cover is locally finite, i.e. each element
of the cover intersects at most finitely many others, then its refinement is again countable.
Local finiteness is crucial, and it is the reason to choose $q$ as large as possible.

\medskip
Contrary to \cite{Sarig-JAMS,Lima-Sarig,Lima-Matheus}, we find no difficulty on the control
of the geometry of $M$ (all exponential maps are identities) neither with the geometry of
stable and unstable directions (the stable direction is trivial). Hence the methods
we use in steps (2)--(5) are more clear and more easily implemented than those in
\cite{Sarig-JAMS,Lima-Sarig,Lima-Matheus}. For example, we do not
make use of graph transforms. On the other hand, a difficulty for the implementation of steps (2)--(5) 
is that neither $\wh M$ nor $\wh f$ are smooth objects. This is not a big issue,
since what we want is to control the action of $f$ and its inverse branches,
and this can be made by controlling the action of $\wh f^{\pm 1}$ in the zeroth coordinate.
Working with natural extensions makes step (1) heavier in notation,
and step (6) more complicated to implement. 


\medskip
Natural extensions have been previously used to investigate nonuniformly expanding systems.
Up to the author's knowledge, the first one to use this
approach was Ledrappier, in the context of absolutely
continuous invariant measures of interval maps \cite{Ledrappier-acip}. Other employments
of this approach are \cite{Gelfert-repellers,Dobbs-equilibrium-measures}.

\medskip
The methods employed in this article require some familiarity with
the articles \cite{Sarig-JAMS,Lima-Sarig,Lima-Matheus}, and a first reading might be
difficulty for those not familiar with the referred literature. Unfortunately,
a self-contained exposition would lead to a lengthy manuscript, thus preventing 
to focus on the novelty of the work.

\subsection{Preliminaries}\label{Section-preliminaries}

Let $\mathfs G=(V,E)$ be an oriented graph, where $V=$ vertex set and $E=$ edge set.
We denote edges by $v\to w$, and we assume that $V$ is countable. 

\medskip
\noindent
{\sc Topological Markov shift (TMS):} A {\em topological Markov shift} (TMS) is a pair $(\Sigma,\sigma)$
where
$$
\Sigma:=\{\text{$\Z$--indexed paths on $\mathfs G$}\}=
\left\{\un{v}=\{v_n\}_{n\in\Z}\in V^{\Z}:v_n\to v_{n+1}, \forall n\in\Z\right\}
$$
and $\sigma:\Sigma\to\Sigma$ is the {\em left shift}, $[\sigma(\un v)]_n=v_{n+1}$. 
The {\em recurrent set} of $\Sigma$ is
$$
\Sigma^\#:=\left\{\un v\in\Sigma:\exists v,w\in V\text{ s.t. }\begin{array}{l}v_n=v\text{ for infinitely many }n>0\\
v_n=w\text{ for infinitely many }n<0
\end{array}\right\}.
$$
We endow $\Sigma$ with the distance $d(\un v,\un w):={\rm exp}[-\min\{|n|:n\in\Z\text{ s.t. }v_n\neq w_n\}]$.
This choice is not canonical, and affects the H\"older regularity of
$\pi$ in Theorem \ref{Thm-Main}.

\medskip
Write $a=e^{\pm\ve}b$ when $e^{-\ve}\leq \frac{a}{b}\leq e^\ve$,
and $a=\pm b$ when $-|b|\leq a\leq |b|$. Given an open set $U\subset \R$ and $h:U\to \R$,
let $\|h\|_0:=\sup_{x\in U}\|h(x)\|$ denote the $C^0$ norm of $h$. For $0<\beta<1$,
let $\Hol{\beta}(h):=\sup\frac{\|h(x)-h(y)\|}{\|x-y\|^\beta}$ 
where the supremum ranges over distinct elements $x,y\in U$.
If $h$ is differentiable, let
$\|h\|_1:=\|h\|_0+\|dh\|_0$ denote its $C^1$ norm, and
$\|h\|_{1+\beta}:=\|h\|_1+\Hol{\beta}(dh)$ its $C^{1+\beta}$ norm.
For $r>0$, define $R[r]:=[-r,r]\subset\R$, where $\R$ is endowed with the usual euclidean distance.

\subsection{Natural extensions}\label{Section-natural-extension}

Most of the discussion below is classical, see e.g. \cite{Rohlin-Exactness} or
\cite[\S 3.1]{Aaronson-book}.
Given a (possibly non-invertible) map $f:M\to M$, let
$$
\wh M:=\{\wh x=(x_n)_{n\in\Z}:f(x_{n-1})=x_n, \forall n\in\Z\}.
$$
Although $\wh M$ does depend on $f$, we do not write this dependence.
Endow $\wh M$ with the distance
$\wh d(\wh x,\wh y):=\sup\{2^nd(x_n,y_n):n\leq 0\}$; then $\wh M$ is
a compact metric space. As for TMS, 
the definition of $\widehat d$ is not canonical and reflects the
H\"older regularity of $\pi$ in Theorem \ref{Thm-Main}.
For each $n\in\Z$, let $\vartheta_n:\wh M\to M$ be the projection into the $n$--th
coordinate, $\vartheta_n[\wh x]=x_n$. Let $\wh{\mathfs B}$ be the 
sigma-algebra in $\wh M$ generated by $\{\vartheta_n:n\leq 0\}$, i.e.
$\wh{\mathfs B}$ is the smallest sigma-algebra that makes all $\vartheta_n$, $n\leq 0$, measurable.

\medskip
\noindent
{\sc Natural extension of $f$:} The {\em natural extension} of $f$ is the 
map $\wh f:\wh M\to\wh M$ defined by
$\wh f(\ldots,x_{-1};x_0,\ldots)=(\ldots,x_0;f(x_0),\ldots)$.
It is an invertible map, with inverse $\wh f^{-1}(\ldots,x_{-1};x_0,\ldots)=(\ldots,x_{-2};x_{-1},\ldots)$.

\medskip
Note that $\wh f$ is indeed an extension
of $f$, since $\vt_0\circ \wh f=f\circ\vt_0$. It is the
smallest invertible extension of $f$: any other invertible extension of $f$ is an 
extension of $\wh f$. The benefit of considering the natural extension is that,
in addition to having an invertible map explicitly defined, its complexity is the same as that of $f$:
there is a bijection between $f$--invariant and $\wh f$--invariant probability measures,
as follows.

\medskip
\noindent
{\sc Projection of a measure:} If $\wh\mu$ is an $\wh f$--invariant probability measure, then
$\mu=\wh\mu\circ \vt_0^{-1}$ is an $f$--invariant probability measure.

\medskip
\noindent
{\sc Lift of a measure:} If $\mu$ is an $f$--invariant probability measure,
let $\wh\mu$ be the unique probability measure on $\wh M$ s.t.
$\wh\mu[\{\wh x\in\wh M:x_n\in A\}]=\mu[A]$ for all $A\subset M$ Borel and all $n\leq 0$.

\medskip
It is clear that $\wh\mu$ is $\wh f$--invariant.
What is less clear is that the projection and lift procedures above are inverse
operations, and that they preserve the Kolmogorov-Sina{\u\i} entropy, see \cite{Rohlin-Exactness}.
Here is one consequence of this fact: $\mu$ is an equilibrium measure for a potential $\varphi:M\to\R$ iff
$\wh\mu$ is an equilibrium measure for $\varphi\circ\vt_0:\wh M\to\R$. 
In particular, the topological entropies of $f$ and $\wh f$ coincide, and 
$\mu$ is a measure of maximal entropy for $f$ iff $\wh\mu$ is a measure
of maximal entropy for $\wh f$.


\medskip
Now let $N=\bigsqcup_{x\in M}N_x$ be a vector bundle over $M$, and let $A:N\to N$ measurable
s.t. for every $x\in M$ the restriction $A\restriction_{N_x}$ is a linear
isomorphism $A_x:N_x\to N_{f(x)}$. For example, if $f$ is a differentiable endomorphism on a manifold $M$,
then we can take $N=TM$ and $A=df$. The map $A$ defines a (possibly non-invertible) cocycle $(A^{(n)})_{n\geq 0}$
over $f$ by $A^{(n)}_x=A_{f^{n-1}(x)}\cdots A_{f(x)}A_x$ for $x\in M$, $n\geq 0$.
There is a way of extending $(A^{(n)})_{n\geq 0}$ to an invertible cocycle over $\wh f$.
For $\wh x\in\wh M$, let $N_{\wh x}:=N_{\vt_0[\wh x]}$ and let
$\wh N:=\bigsqcup_{\wh x\in\wh M} N_{\wh x}$, a vector bundle over $\wh M$.
Define the map $\wh A:\wh N\to\wh N$, $\wh A_{\wh x}:=A_{\vt_0[\wh x]}$.
For $\wh x=(x_n)_{n\in\Z}$, define 
$$
\wh A^{(n)}_{\wh x}:=
\left\{
\begin{array}{ll}
A^{(n)}_{x_0}&,\text{ if }n\geq 0\\
A_{x_{-n}}^{-1}\cdots A_{x_{-2}}^{-1}A_{x_{-1}}^{-1}&,\text{ if }n\leq 0.
\end{array}
\right.
$$
By definition,
$\wh A^{(m+n)}_{\wh x}=\wh A^{(m)}_{\wh f^n(\wh x)}\wh A^{(n)}_{\wh x}$
for all $\wh x\in\wh M$ and all $m,n\in\Z$, hence $(\wh A^{(n)})_{n\in\Z}$
is an invertible cocycle over $\wh f$.

\medskip
Whenever it is convenient, we will write $\vt$ to represent $\vartheta_0$.

\section{Pesin theory}

We define changes of coordinates for which the inverse branches of $f$ become uniformly contracting.
Fix $\chi>0$.

\medskip
\noindent
{\sc The set ${\rm NUE}_\chi$:} It is the set of points 
$\wh x\in \wh M\backslash \bigcup_{n\in\Z}\wh f^n(\vartheta^{-1}[\mathfs S])$ s.t.
$$
\lim_{n\to\pm\infty}\tfrac{1}{n}\log\|\wh{df}^{(n)}_{\wh x}\|>\chi.
$$

\medskip
\noindent
{\sc Parameter $u(\wh x)$:} For $\wh{x}\in{\rm NUE}_\chi$, define
$$
u(\wh{x}):=\left(\sum_{n\geq 0}e^{2n\chi}|\wh{df}^{(-n)}_{\wh x}|^2\right)^{1/2}.
$$

\medskip
It is clear that $1<u(\wh{x})<\infty$. The parameter $u(\wh x)$ controls the quality of hyperbolicity (contraction)
of the inverse branches of $f$.

\medskip
\noindent
{\sc Pesin chart $\Psi_{\wh x}$:} For $\wh x\in{\rm NUE}_\chi$, the {\em Pesin chart} at $\wh{x}$
is the map $\Psi_{\wh x}:\R\to\R$,
$\Psi_{\wh x}(t):=\tfrac{1}{u(\wh x)}t+\vartheta[\wh x]$.
It is a diffeomorphism with $\Psi_{\wh x}(0)=\vt[\wh x]$ and $d\Psi_{\wh x}\equiv\tfrac{1}{u(\wh x)}$.

\begin{lemma}\label{Lemma-linear-reduction}
For all $\wh x\in{\rm NUE}_\chi$, the composition
$F_{\wh x}:=\Psi_{\wh f(\wh x)}^{-1}\circ f\circ\Psi_{\wh x}$
is a diffeomorphism from $R[2\mathfrak r(x_0)]$ onto its image,
and it satisfies $F_{\wh x}(0)=0$ and $|(dF_{\wh x})_0|>e^{\chi}$.
\end{lemma}

\begin{proof}
Let $\wh x=(x_n)_{n\in\Z}$. The maps $\Psi_{\wh f(\wh x)}^{-1},\Psi_{\wh x}$ are globally defined
diffeomorphisms. Observing that $\Psi_{\wh x}(R[2\mathfrak r(x_0)])\subset B(x_0,2\mathfrak r(x_0))$ and that, 
by (A1), the restriction of $f$ to $B(x_0,2\mathfrak r(x_0))$ is a diffeomorphism onto its image,
it follows that $F_{\wh x}$ is a diffeomorphism from $R[2\mathfrak r(x_0)]$ onto its image.
Also, $F_{\wh x}(0)=(\Psi_{\wh f(\wh x)}^{-1}\circ f)(\vt[\wh x])=\Psi_{\wh f(\wh x)}^{-1}(\vt[\wh f(\wh x)])=0$.
It remains to estimate $|(dF_{\wh x})_0|$. Write $x=\vartheta[\wh x]$ and note that
$(dF_{\wh x})_0=df_x\tfrac{u(\wh f(\wh x))}{u(\wh x)}$.
Since $\wh{df}_{\wh x}^{(-n)}=\wh{df}^{(-n-1)}_{\wh f(\wh x)}\circ \wh{df}_{\wh x}=
\wh{df}^{(-n-1)}_{\wh f(\wh x)}\circ df_{x}$
for $n\geq 0$, we have
\begin{align*}
&\, u(\wh f(\wh x))^2=1+\sum_{n\geq 0}e^{2(n+1)\chi}|\wh{df}^{(-n-1)}_{\wh f(\wh x)}|^2=
1+e^{2\chi}|df_x|^{-2}\sum_{n\geq 0}e^{2n\chi}|\wh{df}^{(-n)}_{\wh x}|^2\\
&=1+e^{2\chi}|df_x|^{-2}u(\wh x)^2,
\end{align*}
therefore
$|(dF_{\wh x})_0|^2=e^{2\chi}\tfrac{u(\wh f(\wh x))^2}{u(\wh f(\wh x))^2-1}>e^{2\chi}$.
\end{proof}

Now we control the distance of trajectories of $\wh f$ to the singular set $\mathfs S$.
For $\wh x\in\wh M$, define $\rho(\wh x):=d\left(\{\vt_{-1}[\wh x],\vt_{0}[\wh x],\vt_{1}[\wh x]\},\mathfs S\right)$.

\medskip
\noindent
{\sc Regular set:} The {\em regular set of $f$} is 
\begin{align*}
{\rm Reg}:=&\left\{\wh x\in  \wh M\backslash \bigcup_{n\in\Z}\wh f^n(\vartheta^{-1}[\mathfs S]):
\lim_{n\to\pm\infty}\tfrac{1}{|n|}\log\rho(\wh f^n(\wh x))=0\right\}\\
=&\left\{\wh x\in  \wh M\backslash \bigcup_{n\in\Z}\wh f^n(\vartheta^{-1}[\mathfs S]):
\lim_{n\to\pm\infty}\tfrac{1}{|n|}\log d(\vt_n[\wh x],\mathfs S)=0\right\}.
\end{align*}

\medskip
\noindent
{\sc The set ${\rm NUE}_\chi^*$:} It is the set of $\wh x\in{\rm NUE}_\chi$ with the following properties:
\begin{enumerate}[(1)]
\item $\wh x\in{\rm Reg}$.
\item There exist a sequence $n_k\to-\infty$ s.t.
$u(\wh f^{n_k}(\wh x))\to u(\wh x)$. 
\item $\lim_{n\to-\infty}\tfrac{1}{n}\log u(\wh f^n(\wh x))=0$.
\end{enumerate}

\medskip
The next lemma shows that ${\rm NUE}_\chi^*$ carries the measures that are relevant to us.


\begin{lemma}\label{Lemma-adaptedness}
If $\wh\mu=$ lift of $f$--adapted $\chi$--expanding measure,
then $\wh\mu[{\rm NUE}_\chi^*]=1$.
\end{lemma}


\begin{proof}
Let $\mu$ be an $f$--adapted $\chi$--expanding measure, and let $\wh\mu$ be its lift.
Let $\wh X_0:=\wh M\backslash\bigcup_{n\in\Z}\wh f^n(\vt^{-1}(\mathfs S))$.
By $f$--adaptedness and the definition of $\wh\mu$ we have
$\wh\mu(\vt^{-1}[\mathfs S])=\mu(\mathfs S)=0$, hence $\wh\mu(\wh X_0)=1$.
By (A2), $|\log |df_x||\leq a|\log d(x,\mathfs S)|$ and so $f$--adaptedness implies that
$\log |df_x|\in L^1(\mu)$. Hence 
$\int |\log|\wh{df}_{\wh x}||d\wh\mu(\wh x)=\int |\log|df_x||d\mu(x)<\infty$, thus proving that
$\log |\wh{df}|\in L^1(\wh\mu)$.
By the Birkhoff ergodic theorem,
$\exists \wh X_1\subset\wh M$ with
$\wh\mu(\wh X_1)=1$ s.t. $\lim_{n\to\pm\infty}\tfrac{1}{n}\log|\wh{df}^{(n)}_{\wh x}|$
exists for all $\wh x\in\wh X_1$. Now let $Y$ be the set of $x\in M$
s.t. $\lim_{n\to+\infty}\tfrac{1}{n}\log|df^n_x|>\chi$, and let $\wh X_2=\vartheta^{-1}(Y)$.
We also have $\wh\mu(\wh X_2)=1$, therefore $\mu(\wh X_0\cap\wh X_1\cap\wh X_2)=1$.
Since $\wh{df}^{(n)}_{\wh x}=df^n_{\vartheta[\wh x]}$ for $n\geq 0$, we have
that $\lim_{n\to\pm\infty}\tfrac{1}{n}\log|\wh{df}^{(n)}_{\wh x}|>\chi$
for all $\wh x\in\wh X_1\cap\wh X_2$, therefore $\wh X_0\cap \wh X_1\cap\wh X_2\subset{\rm NUE}_\chi$.
This implies that $\wh\mu[{\rm NUE}_\chi]=1$.

\medskip
To prove that $\wh\mu[{\rm NUE}_\chi^*]=1$, it remains to show that
conditions (1)--(3) hold $\wh\mu$--a.e.
We have $\int |\log d(\vt[\wh x],\mathfs S)|d\wh\mu(\wh x)=\int |\log d(x,\mathfs S)|d\mu(x)<\infty$.
By $\wh f$--invariance,
$\log d(\vt_{-1}[\wh x],\mathfs S),\log d(\vt_0[\wh x],\mathfs S),
\log d(\vt_1[\wh x],\mathfs S)$ are in $L^1(\mu)$,
hence is also their minimum $\log\rho(\wh x)$. By
the Birkhoff ergodic theorem\footnote{Here we are using that if $\varphi:M\to\R$ satisfies
$\int |\varphi|d\mu<\infty$ then $\lim_{n\to\pm\infty}\tfrac{1}{n}\varphi(f^n(x))=0$ $\mu$--a.e.
Indeed, by the Birkhoff ergodic theorem
$\widetilde\varphi(x)=\lim_{n\to\infty}\tfrac{1}{n}\sum_{i=0}^{n-1}\varphi(f^i(x))$ exists $\mu$--a.e., hence
$\lim_{n\to\infty}\tfrac{1}{n}\varphi(f^n(x))=\lim_{n\to\infty}\left[\tfrac{1}{n}\sum_{i=0}^{n}\varphi(f^i(x))
-\tfrac{1}{n}\sum_{i=0}^{n-1}\varphi(f^i(x))\right]=0$ $\mu$--a.e. The same argument works
for $n\to-\infty$.} we get $\wh\mu({\rm Reg})=1$. By the Poincar\'e recurrence theorem,
condition (2) holds $\wh\mu$--a.e.

\medskip
Finally, we check that condition (3) holds $\wh\mu$--a.e. For $\wh x\in{\rm NUE}_\chi$,
let $\varphi(\wh x):=\log|(dF_{\wh x})_0|$.
It is enough to show that $\varphi\in L^1(\wh\mu)$, because of the following:
\begin{enumerate}[$\circ$]
\item By the proof of Lemma \ref{Lemma-linear-reduction},
$\varphi=\psi+[\log u\circ \wh f-\log u]$,
where $\psi:=\log|\wh{df}|\in L^1(\wh\mu)$.
\item By the Poincar\'e recurrence theorem, $\liminf_{n\to-\infty}\tfrac{1}{n}\log u(\wh f^n(\wh x))=0$
$\wh\mu$--a.e., hence
$\liminf_{n\to-\infty}\tfrac{\varphi_n(\wh x)}{n}=\liminf_{n\to-\infty}\tfrac{\psi_n(\wh x)}{n}$
$\wh\mu$--a.e., where $\varphi_n,\psi_n$ denote the Birkhoff sums of $\varphi,\psi$ with respect to $\wh f$.
\item By the Birkhoff ergodic theorem,
$\lim_{n\to-\infty}\tfrac{\varphi_n(\wh x)}{n}=\lim_{n\to-\infty}\tfrac{\psi_n(\wh x)}{n}$
$\wh\mu$--a.e., hence $\lim_{n\to-\infty}\tfrac{1}{n}\log u(\wh f^n(\wh x))=0$
$\wh\mu$--a.e.
\end{enumerate}
We show that $\varphi\in L^1(\wh\mu)$. By Lemma \ref{Lemma-linear-reduction}
we have $\varphi>\chi$, hence it is enough to prove that $\int\varphi d\wh\mu<\infty$.
Since $2\varphi(\wh x)=\log\left(e^{2\chi}\tfrac{u(\wh f(\wh x))^2}{u(\wh f(\wh x))^2-1}\right)$
and $\wh\mu$ is $\wh f$--invariant, this former claim is equivalent to showing that
$\int\log\left(e^{-2\chi}\tfrac{u(\wh x)^2-1}{u(\wh x)^2}\right)>-\infty$.
By (A2),
$u(\wh x)^2\geq 1+e^{2\chi}\|\wh{df}^{(-1)}_{\wh x}\|^2\geq 
1+e^{2\chi}d(\vartheta_{-1}[\wh x],\mathfs S)^{2a}\geq 1+e^{2\chi}\rho(\wh x)^{2a}$
thus
$$
e^{-2\chi}\frac{u(\wh x)^2-1}{u(\wh x)^2}
=e^{-2\chi}\left(1-\frac{1}{u(\wh x)^2}\right)
\geq\frac{\rho(\wh x)^{2a}}{1+e^{2\chi}\rho(\wh x)^{2a}}\geq\frac{\rho(\wh x)^{2a}}{1+e^{2\chi}}\,,
$$
hence
$
\int \log\left(e^{-2\chi}\tfrac{u(\wh x)^2-1}{u(\wh x)^2}\right)d\wh\mu(\wh x)
\geq 2a\int\log \rho(\wh x)d\wh\mu(\wh x)-\log(1+e^{2\chi})>-\infty.
$
This completes the proof of the lemma.
\end{proof}

\subsection{Inverse branches}

By (A1), the inverse branch of $f$ that sends $f(x)$ to $x$ is a well-defined
diffeomorphism from $E_x=B(f(x),2\mathfrak r(x))$
onto its image.

\medskip
\noindent
{\sc Inverse branch of $f$:} Let $g_x:E_x\to g_x(E_x)$ be the inverse
branch of $f$ that sends $f(x)$ to $x$. For $\wh x=(x_n)_{n\in\Z}\in\wh M$,
define $g_{\wh x}:=g_{x_{-1}}$.

\medskip
Note that $g_{\wh x}(x_0)=x_{-1}$.
We want to mimic the behavior of $(dF_{\wh x})_0$ to the inverse branches $g_{\wh x}$.
For that, we need to reduce the domains of Pesin charts to intervals so that their images
do not intersect the singular set $\mathfs S$ and at the same time we can control the variation of $df$.
Given $\ve>0$, let $I_\ve:=\{e^{-\frac{1}{3}\ve n}:n\geq 0\}$.

\medskip
\noindent
{\sc Parameter $Q_\ve(\wh x)$:} For $\wh x\in{\rm NUE}_\chi$, let
$Q_\ve(\wh x):=\max\{q\in I_\ve:q\leq \widetilde Q_\ve(\wh x)\}$, where
$$
\widetilde Q_\ve(\wh x)=\ve^{3/\beta}
\min\left\{u(\wh x)^{-24/\beta},u(\wh f^{-1}(\wh x))^{-12/\beta}\rho(\wh x)^{72a/\beta}\right\}.
$$

\medskip
The term $\ve^{3/\beta}$ will allow to absorb multiplicative constants.
The choice of $Q_\ve(\wh x)$ guarantees that the inverse of $F_{\wh f^{-1}(\wh x)}$
is well-defined in $R[10Q_\ve(\wh x)]$
and that it is close to a linear contraction (Theorem \ref{Thm-non-linear-Pesin}),
and it also allows to compare nearby Pesin charts (Proposition \ref{Lemma-overlap}).
We have the following trivial bounds:
\begin{align*}
&Q_\ve(\wh x)\leq \ve^{3/\beta}, u(\wh x)Q_\ve(\wh x)^{\beta/24}\leq \ve^{1/8},
u(\wh f^{-1}(\wh x))Q_\ve(\wh x)^{\beta/12}\leq \ve^{1/4},\\
&\rho(\wh x)^{-a}Q_\ve(\wh x)^{\beta/72}\leq \ve^{1/24}.
\end{align*}

Let $\wh x=(x_n)_{n\in\Z}\in{\rm NUE}_\chi$. The definition of $Q_\ve(\wh x)$
is strong enough to guarantee that $\Psi_{\wh x}(R[10Q_\ve(\wh x)])$ is contained 
in neighborhoods where (A1)--(A3) hold:
\begin{enumerate}[$\circ$]
\item $\Psi_{\wh x}(R[10Q_\ve(\wh x)])\subset D_{x_0}\cap E_{x_{-1}}$: we have
$\Psi_{\wh x}(R[10Q_\ve(\wh x)])\subset B(x_0,10Q_\ve(\wh x))\subset D_{x_0}\cap E_{x_{-1}}$,
since $10Q_\ve(\wh x)<10\ve^{3/\beta}\rho(\wh x)^a<\mathfrak r(x_{-1}),\mathfrak r(x_0)$.
\item $g_{\wh x}(\Psi_{\wh x}(R[10Q_\ve(\wh x)]))\subset D_{x_{-1}}$: by the previous item,
$g_{\wh x}(\Psi_{\wh x}(R[10Q_\ve(\wh x)]))$ is well-defined. By (A2),
\begin{align*}
&\ g_{\wh x}(\Psi_{\wh x}(R[10Q_\ve(\wh x)]))\subset g_{\wh x}[B(x_0,10Q_\ve(\wh x))]\\
&\subset B(x_{-1},10Q_\ve(\wh x)d(x_{-1},\mathfs S)^{-a})\subset D_{x_{-1}},
\end{align*}
since
$10Q_\ve(\wh x)d(x_{-1},\mathfs S)^{-a}<10\ve^{3/\beta}\rho(\wh x)^{2a}d(x_{-1},\mathfs S)^{-a}
<\rho(\wh x)^a<\mathfrak r(x_{-1})$.
\item $\Psi_{\wh x}(R[10Q_\ve(\wh x)])\subset g_{\wh f(\wh x)}(E_{x_0})$: noting that
$g_{\wh f(\wh x)}=g_{x_0}$, assumption (A2) implies 
$g_{x_0}(E_{x_0})\supset B(x_0,2\mathfrak r(x_0)d(x_0,\mathfs S)^a)\supset B(x_0,10Q_\ve(\wh x))$,
since $10Q_\ve(\wh x)<10\ve^{3/\beta}\rho(\wh x)^{2a}<2\mathfrak r(x_0)d(x_0,\mathfs S)^a$.
Thus $g_{x_0}(E_{x_0})\supset \Psi_{\wh x}(R[10Q_\ve(\wh x)])$.
\end{enumerate}
The third item implies the following: if $y\in \Psi_{\wh x}(R[10Q_\ve(\wh x)])$,
then $y$ is the unique pre-image of $f(y)$ in $g_{\wh f(\wh x)}(E_{x_0})$, and $y=g_{\wh f(\wh x)}(f(y))$.

\begin{lemma}[Tempering Kernel]\label{Lemma-temperedness}
If $\wh x\in{\rm NUE}_\chi^*$, then
$$
\lim_{n\to-\infty}\tfrac{1}{|n|}\log Q_\ve(\wh f^n(\wh x))=0.
$$
\end{lemma}

\begin{proof}
Clearly $\limsup_{n\to-\infty}\tfrac{1}{|n|}\log Q_\ve(\wh f^n(\wh x))\leq 0$.
Conversely, $\wh x\in{\rm Reg}$ implies that $\lim_{n\to-\infty}\tfrac{1}{|n|}\log\rho(\wh f^n(\wh x))=0$.
By condition (3) in the definition of ${\rm NUE}_\chi^*$,
$\lim_{n\to-\infty}\tfrac{1}{|n|}\log u(\wh f^n(\wh x))=0$,
therefore $\liminf_{n\to-\infty}\tfrac{1}{|n|}\log Q_\ve(\wh f^n(\wh x))\geq 0$.
\end{proof}

\medskip
\noindent
{\sc Inverse branches of $f$ in Pesin charts:} For $\wh x\in{\rm NUE}_\chi$, let
$G_{\wh x}:R[2\mathfrak r(x_{-1})]\to\R$ be the composition defined by
$G_{\wh x}:= \Psi_{\wh f^{-1}(\wh x)}^{-1}\circ g_{\wh x}\circ\Psi_{\wh x}$.

\medskip
If $\wh x=(x_n)_{n\in\Z}$ then $G_{\wh x}$ is the representation of $g_{\wh x}$ in Pesin charts.
Note that $G_{\wh x}$ is a diffeomorphism from $R[2\mathfrak r(x_{-1})]$ onto its image, since 
$\Psi_{\wh f^{-1}(\wh x)}^{-1},\Psi_{\wh x}$ are globally defined diffeomorphisms and $g_{\wh x}$
is a diffeomorphism from $E_{x_{-1}}\supset \Psi_{\wh x}(R[2\mathfrak r(x_{-1})])$ onto its image, 
by (A1). Also, $G_{\wh x}$ is the inverse of $F_{\wh f^{-1}(\wh x)}$ where the compositions
are well-defined.
The next theorem gives a better understanding of $G_{\wh x}$.

\begin{theorem}\label{Thm-non-linear-Pesin} 
The following holds for all $\ve>0$ small enough: If $\wh x\in{\rm NUE}_\chi$
then $G_{\wh x}$ is a diffeomorphism from $R[10Q_\ve(\wh x)]$
onto its image, and it can be written as $G_{\wh x}(t)=At+h(t)$ where:
\begin{enumerate}[{\rm (1)}]
\item $|A|<e^{-\chi}$.
\item $h(0)=dh_0=0$ and $\|h\|_{1+\beta/2}<\ve$, where the norm is taken in $R[10Q_\ve(\wh x)]$.
\end{enumerate}
In particular, $\|dG_{\wh x}\|_0< e^{-\chi/2}$.
\end{theorem}

\begin{proof}
Write $\wh x=(x_n)_{n\in\Z}$. Since $10Q_\ve(\wh x)<\rho(\wh x)^a<2\mathfrak r(x_{-1})$,
it follows from the previous paragraph that the restriction of $G_{\wh x}$ to $R[10Q_\ve(\wh x)]$
is a diffeomorphism onto its image.
Now we check (1)--(2). Let $A:=(dG_{\wh x})_0$ and
$h:R[10Q_\ve(\wh x)]\to\R$ s.t. $G_{\wh x}(t)=At+h(t)$.
By Lemma \ref{Lemma-linear-reduction}, $|A|=|(dF_{\wh f^{-1}(\wh x)})_0|^{-1}<e^{-\chi}$.
Clearly $h(0)=h'(0)=0$, so it remains to estimate $\|h\|_{1+\beta/2}$.

\medskip
\noindent
{\sc Claim:} $|(dG_{\wh x})_{t_1}-(dG_{\wh x})_{t_2}|\leq \tfrac{\ve}{3}|t_1-t_2|^{\beta/2}$
for all $t_1,t_2\in R[10Q_\ve(\wh x)]$.

\medskip
Before proving the claim, let us show how to conclude (2).
If $\ve>0$ is small enough then $R[10Q_\ve(\wh x)]\subset R[1]$. Applying the claim with $t_2=0$,
we get $|h'(t)|\leq \frac{\ve}{3}|t|^{\beta/2}<\tfrac{\ve}{3}$. By the mean value inequality,
$|h(t)|\leq\tfrac{\ve}{3}|t|<\tfrac{\ve}{3}$, hence $\|h\|_{1+\beta/2}<\ve$.

\begin{proof}[Proof of the claim.]
We have $\Psi_{\wh x}(t_1),\Psi_{\wh x}(t_2)\in E_{x_{-1}}$. Using that $\Psi_{\wh x}$
is a contraction and assumption (A3), we get:
\begin{align*}
|(dG_{\wh x})_{t_1}-(dG_{\wh x})_{t_2}|=
\tfrac{u(\wh f^{-1}(\wh x))}{u(\wh x)}|g'(\Psi_{\wh x}(t_1))-g'(\Psi_{\wh x}(t_2))|\leq
\mathfrak K u(\wh f^{-1}(\wh x))|t_1-t_2|^\beta.
\end{align*}
Since $|t_1-t_2|<20Q_\ve(\wh x)$, if $\ve>0$ is small then:
\begin{align*}
\mathfrak K u(\wh f^{-1}(\wh x))|t_1-t_2|^{\beta/2}<
20\mathfrak K u(\wh f^{-1}(\wh x))\ve^{3/2}u(\wh f^{-1}(\wh x))^{-6}<20\mathfrak K\ve^{3/2}<\ve.
\end{align*}
This completes the proof of the claim.
\end{proof}

If $\ve>0$ is small enough then $\|dG_{\wh x}\|_0\leq |(dG_{\wh x})_0|+\|dh\|_0<e^{-\chi}+\ve<e^{-\chi/2}$.
\end{proof}

\subsection{The overlap condition}\label{section-overlap}

Our next goal is to identify when two Pesin charts $\Psi_{\wh x},\Psi_{\wh y}$
are close. Even when $\vt[\wh x]$ and $\vt[\wh y]$ are nearby points of $M$,
the distortions of $\Psi_{\wh x}$ and $\Psi_{\wh y}$ might be very different. 
The values controlling such distortions are $u(\wh x)$ and $u(\wh y)$,
so we need to compare them. Another requirement for Pesin charts to be close
is that their domains of definition have comparable sizes. Because of this, we 
consider Pesin charts with different domains, which we call {\em $\ve$--charts}.

\medskip
\noindent
{\sc $\ve$--chart:} An {\em $\ve$--chart} $\Psi_{\wh x}^p$ is the restriction
$\Psi_{\wh x}\restriction_{[-p,p]}$, where $0<p\leq Q_\ve(\wh x)$.

\medskip
Note that for each $\wh x\in{\rm NUE}_\chi$ there are infinitely many
$\ve$--charts centered at $\wh x$. 

\medskip
\noindent
{\sc $\ve$--overlap:} Two $\ve$--charts $\Psi_{\wh x_1}^{p_1},\Psi_{\wh x_2}^{p_2}$
are said to {\em $\ve$--overlap} if $\tfrac{p_1}{p_2}=e^{\pm\ve}$ and
$$
d(\vartheta[\wh x_1],\vartheta[\wh x_2])+|u(\wh x_1)^{-1}-u(\wh x_2)^{-1}|<(p_1p_2)^4.
$$
When this happens, we write $\Psi_{\wh x_1}^{p_1}\overset{\ve}{\approx}\Psi_{\wh x_2}^{p_2}$.

\medskip
Clearly, if $\Psi_{\wh x_1}^{p_1}\overset{\ve}{\approx}\Psi_{\wh x_2}^{p_2}$
then $\Psi_{\wh x_1}^{cp_1}\overset{\ve}{\approx}\Psi_{\wh x_2}^{cp_2}$
for all $c>1$ s.t. $cp_i\leq Q_\ve(\wh x_i)$.
The next proposition shows that $\ve$--overlap is strong enough to guarantee
that the Pesin charts are close.

\begin{proposition}\label{Lemma-overlap}
The following holds for $\ve>0$ small enough.
If $\Psi_{\wh x_1}^{p_1}\overset{\ve}{\approx}\Psi_{\wh x_2}^{p_2}$ then:
\begin{enumerate}[{\rm (1)}]
\item {\sc Control of $u$:} $\frac{u(\wh x_1)}{u(\wh x_2)}=e^{\pm(p_1p_2)^3}$.
\item {\sc Overlap:} $\Psi_{\wh x_i}(R[e^{-2\ve}p_i])\subset \Psi_{\wh x_j}(R[p_j])$ for $i,j=1,2$.
\item {\sc Change of coordinates:} For $i,j=1,2$ it holds
$\|\Psi_{\wh x_i}^{-1}\circ\Psi_{\wh x_j}-{\rm Id}\|_{2}<\ve(p_1 p_2)^3$
where the norm is taken in $R[1]$.
\end{enumerate}
\end{proposition}

\begin{proof} (1) Since $\ve>0$ is small, it is enough to prove that
$\left|\tfrac{u(\wh x_1)}{u(\wh x_2)}-1\right|<\ve^{3/\beta}(p_1p_2)^3$.
By assumption, $|u(\wh x_1)^{-1}-u(\wh x_2)^{-1}|<(p_1p_2)^4$.
Also $u(\wh x_1)<\tfrac{\ve^{3/\beta}}{Q_\ve(\wh x_1)}<\tfrac{\ve^{3/\beta}}{p_1p_2}$,
therefore
$$
\left|\tfrac{u(\wh x_1)}{u(\wh x_2)}-1\right|=
u(\wh x_1)|u(\wh x_1)^{-1}-u(\wh x_2)^{-1}|<\ve^{3/\beta}(p_1p_2)^3.
$$

\medskip
\noindent
(2) We prove that $\Psi_{\wh x_1}(R[e^{-2\ve}p_1])\subset \Psi_{\wh x_2}(R[p_2])$.
If $t\in R[e^{-2\ve}p_1]$ then
\begin{align*}
d(\Psi_{\wh x_1}(t),\Psi_{\wh x_2}(t))\leq
d(\vt[\wh x_1],\vt[\wh x_2])+|u(\wh x_1)^{-1}-u(\wh x_2)^{-1}|<(p_1p_2)^4,
\end{align*}
therefore $\Psi_{\wh x_1}(t)\in B(\Psi_{\wh x_2}(t),(p_1p_2)^4)$. We have
$B(\Psi_{\wh x_2}(t),(p_1p_2)^4)\subset \Psi_{\wh x_2}(B)$ where
$B=B(t,u(\wh x_2)(p_1p_2)^4)$. If $t'\in B$ then
$|t'|\leq |t|+u(\wh x_2)(p_1p_2)^4\leq (e^{-\ve}+\ve^{3/\beta})p_2<p_2$
for $\ve>0$ small enough, thus $B\subset R[p_2]$.

\medskip
\noindent
(3) By direct calculation,
$$
(\Psi_{\wh x_2}^{-1}\circ \Psi_{\wh x_1}-{\rm Id})(t)=
\left(\tfrac{u(\wh x_2)}{u(\wh x_1)}-1\right)t+u(\wh x_2)(\vt[\wh x_1]-\vt[\wh x_2])
$$
is a linear function. By part (1), the $C^2$ norm taken in $R[1]$ is 
\begin{align*}
&\, \|\Psi_{\wh x_2}^{-1}\circ \Psi_{\wh x_1}-{\rm Id}\|_{2}=
\|\Psi_{\wh x_2}^{-1}\circ \Psi_{\wh x_1}-{\rm Id}\|_{1}
\leq 2\left|\tfrac{u(\wh x_2)}{u(\wh x_1)}-1\right|+
u(\wh x_2)d(\vt[\wh x_1],\vt[\wh x_2])\\
&<2\ve^{3/\beta}(p_1p_2)^3+\ve^{3/\beta}(p_1p_2)^3<\ve(p_1p_2)^3.
\end{align*}
\end{proof}

\subsection{The map $G_{\wh x,\wh y}$}

Let $\wh x,\wh y\in{\rm NUE}_\chi$, and assume that
$\Psi_{\wh f^{-1}(\wh x)}^p\overset{\ve}{\approx}\Psi_{\wh y}^q$.
We want to change $\Psi_{\wh f^{-1}(\wh x)}$ by $\Psi_{\wh y}$ in
$G_{\wh x}$ and obtain a result similar to Theorem \ref{Thm-non-linear-Pesin}.

\medskip
\noindent
{\sc The map $G_{\wh x,\wh y}$:}
If $\Psi_{\wh f^{-1}(\wh x)}^p\overset{\ve}{\approx}\Psi_{\wh y}^q$,
let $G_{\wh x,\wh y}=\Psi_{\wh y}^{-1}\circ g_{\wh x}\circ \Psi_{\wh x}$
wherever this composition is well-defined.

\medskip
Note that $G_{\wh x,\wh y}$ is the representation of $g_{\wh x}$ in the charts
$\Psi_{\wh x}$ and $\Psi_{\wh y}$. Alternatively, by Proposition \ref{Lemma-overlap},
$G_{\wh x,\wh y}:=\Psi_{\wh y}^{-1}\circ\Psi_{\wh f^{-1}(\wh x)}\circ G_{\wh x}$
is a small perturbation of $G_{\wh x}$. The next result makes this claim more precise. 


\begin{theorem}\label{Thm-non-linear-Pesin-2}
The following holds for all $\ve>0$ small enough:
If $\wh x,\wh y\in{\rm NUE}_\chi$ and
$\Psi_{\wh f^{-1}(\wh x)}^p\overset{\ve}{\approx}\Psi_{\wh y}^q$, then
$G_{\wh x,\wh y}$ is well-defined in $R[10Q_\ve(\wh x)]$ and can be written as
$G_{\wh x,\wh y}(t)=At+h(t)$ where:
\begin{enumerate}[{\rm (1)}]
\item $|A|<e^{-\chi}$, cf. Theorem \ref{Thm-non-linear-Pesin}.
\item $|h(0)|<\ve (pq)^3$, $|dh_0|<\ve (pq)^3$, and
$\Hol{\beta/3}(dh)<\ve$ where the norm is taken in $R[10Q_\ve(\wh x)]$.
\end{enumerate}
In particular, $G_{\wh x,\wh y}$ contracts at least by a factor of $e^{-\chi/2}$.
\end{theorem}

\begin{proof}
We write
$G_{\wh x,\wh y}=:H\circ G_{\wh x}$
and see $G_{\wh x,\wh y}$ as a small perturbation of $G_{\wh x}$.
By Theorem \ref{Thm-non-linear-Pesin},
$$
G_{\wh x}(0)=0,\ \|dG_{\wh x}\|_0<1,\ |(dG_{\wh x})_{t_1}-(dG_{\wh x})_{t_2}|\leq \ve|t_1-t_2|^{\beta/2}
\text{ for } t_1,t_2\in R[10Q_\ve(\wh x)]
$$
where the $C^0$ norm is taken in $R[10Q_\ve(\wh x)]$,
and by Proposition \ref{Lemma-overlap}(3) the function $H$ is affine with
$$
\|H-{\rm Id}\|_0<\ve(pq)^3,\ \|d(H-{\rm Id})\|_0<\ve(pq)^3 
$$
where the $C^0$ norms are taken in $R[1]$.

\medskip
It is easy to show that $G_{\wh x,\wh y}$ is well-defined in $R[10Q_\ve(\wh x)]$:
since $G_{\wh x}(R[10Q_\ve(\wh x)])\subset B(0,10Q_\ve(\wh x))\subset R[1]$,
Proposition \ref{Lemma-overlap}(3) implies that $G_{\wh x,\wh y}$ is well-defined.
Now we prove (1)--(2). Define $h:=G_{\wh x,\wh y}-(dG_{\wh x})_0=H\circ G_{\wh x}- (dG_{\wh x})_0$,
where $(dG_{\wh x})_0$ represents the linear functional on $\R$ defined by the derivative.
Then $|h(0)|=|H(0)|<\ve(pq)^3$
and $|dh_0|=|dH_0-1||(dG_{\wh x})_0|<\ve(pq)^3$.
Finally,
if $\ve>0$ is small enough then for all $t_1,t_2\in R[10Q_\ve(\wh x)]$ we have
\begin{align*}
&|dh_{t_1}-dh_{t_2}|=|dH_{G_{\wh x}(t_1)}(dG_{\wh x})_{t_1}-dH_{G_{\wh x}(t_2)}(dG_{\wh x})_{t_2}|
=\|dH\|_0|(dG_{\wh x})_{t_1}-(dG_{\wh x})_{t_2}|\\
&\leq 2\ve|t_1-t_2|^{\beta/2}<\ve|t_1-t_2|^{\beta/3}.
\end{align*}
This completes the proof of (2). In particular, if $\ve>0$ is 
small enough then $\|dh\|_0\leq \ve(pq)^3+\ve(10Q_\ve(\wh x))^{\beta/3}<\ve$ and hence
$|G_{\wh x,\wh y}(t_1)-G_{\wh x,\wh y}(t_2)|\leq (|A|+\|dh\|_0)|t_1-t_2|\leq (e^{-\chi}+\ve)|t_1-t_2|\leq e^{-\chi/2}|t_1-t_2|$
for all $t_1,t_2\in R[10Q_\ve(\wh x)]$.
\end{proof}

\subsection{$\ve$--generalized pseudo-orbits and the parameter $q_\ve(\wh x)$}

Now we define when we can pass from one $\ve$--chart to another via the action of $\wh f^{-1}$.
We will define two such notions, one weak and one strong. While in
\cite{Sarig-JAMS,Lima-Sarig,Lima-Matheus} the authors only define one notion (similar
to the strong notion presented below), here we also require a weaker one that will be 
relevant for us in Section \ref{Section-symbolic-dynamics}.

\medskip
\noindent
{\sc Weak edge $v\overset{\ve}{\dashleftarrow}w$:} Given $\ve$--charts $v=\Psi_{\wh y}^q$
and $w=\Psi_{\wh x}^p$, we draw a weak edge from $w$ to $v$ if:
\begin{enumerate}
\item[(WE1)] {\sc Overlap:} $\Psi_{\wh f^{-1}(\wh x)}^q\overset{\ve}{\approx}\Psi_{\wh y}^q$.
\item[(WE2)] {\sc Control of parameters:} $p\leq e^\ve q$.
\end{enumerate}
When this happens, we write $v\overset{\ve}{\dashleftarrow}w$. 

\medskip
Clearly, if $\Psi_{\wh y}^q\overset{\ve}{\dashleftarrow}\Psi_{\wh x}^p$ then
$\Psi_{\wh y}^{cq}\overset{\ve}{\dashleftarrow}\Psi_{\wh x}^{cp}$ for all $c>1$ s.t.
$cq\leq Q_\ve(\wh y)$ and $cp\leq Q_\ve(\wh x)$.
For $\ve>0$ small, define $\delta_\ve:=e^{-\ve n}\in I_\ve$
where $n$ is the unique positive integer s.t. $e^{-\ve n}<\ve\leq e^{-\ve(n-1)}$.
In particular, $\delta_\ve<\ve$.

\medskip
\noindent
{\sc Edge $v\overset{\ve}{\leftarrow}w$:} Given $\ve$--charts $v=\Psi_{\wh y}^q$
and $w=\Psi_{\wh x}^p$, we draw an edge from $w$ to $v$ if the following holds:
\begin{enumerate}[.......]
\item[(E1)] {\sc Overlap:} $\Psi_{\wh f^{-1}(\wh x)}^q\overset{\ve}{\approx}\Psi_{\wh y}^q$.
\item[(E2)] {\sc Control of parameters:}
\begin{enumerate}[i .....]
\item[(E2.1)] $d(\vt_1[\wh y],\vt_0[\wh x])<q$.
\item[(E2.2)] $\tfrac{u(\wh f(\wh y))}{u(\wh x)}=e^{\pm q}$.
\item[(E2.3)] $p=\min\{e^\ve q,\delta_\ve Q_\ve(\wh x)\}$.
\end{enumerate}
\end{enumerate}
When this happens, we write $v\overset{\ve}{\leftarrow}w$.

\medskip
It is not hard to see that condition (E2.1) follows from (E1) and assumption (A2), but for reference purposes
we write it separately. The parameters $p,q$ are the sizes of unstable manifolds
in the charts, and the greedy recursion in (E2.3) implies that, fixed an unstable
manifold at $\wh y$, the unstable manifold at $\wh x$ is as big as possible.
This maximality is crucial to prove the
inverse theorem (Theorem \ref{Thm-inverse}).

\begin{remark}\label{Remark-overlap}
Since $f$ is nonuniformly expanding, our definition of edge is different from
those in \cite{Sarig-JAMS,Lima-Sarig,Lima-Matheus} in two senses.
On one hand, we only need to consider one overlap 
and one recursive relation. On the other hand, the {\em lack of symmetry}
between $f$ and its inverse requires us to control some parameters separately,
as stated in (E2.1) and (E2.2).
\end{remark}

\begin{lemma}\label{Lemma-edge}
The following holds for all $\ve>0$ small enough. If $\Psi_{\wh y}^q\overset{\ve}{\dashleftarrow}\Psi_{\wh x}^p$ then:
\begin{enumerate}[{\rm (1)}]
\item $G_{\wh x,\wh y}(R[p])\subset R[q]$. 
\item If $x\in \Psi_{\wh x}(R[p])$ then $g_{\wh x}(x)$ is the unique $y\in \Psi_{\wh y}(R[q])$
s.t. $f(y)=x$.
\end{enumerate}
\end{lemma}

\begin{proof}
By Theorem \ref{Thm-non-linear-Pesin-2}
and (WE2), $G_{\wh x,\wh y}(R[p])\subset B(G_{\wh x,\wh y}(0),e^{-\frac{\chi}{2}}p)\subset
R[\ve q^6+e^{-\frac{\chi}{2}}p]\subset R[q]$, since
$\ve q^6+e^{-\frac{\chi}{2}}p<\ve q+e^{-\frac{\chi}{2}+\ve}q=(\ve+e^{-\frac{\chi}{2}+\ve})q<q$
for $\ve>0$ small enough. This proves part (1). Now take $x=\Psi_{\wh x}(t)\in \Psi_{\wh x}(R[p])$
and let $y=g_{\wh x}(x)=(\Psi_{\wh y}\circ G_{\wh x,\wh y})(t)$. By definition $f(y)=x$, and by part (1)
it holds $y\in  \Psi_{\wh y}(R[q])$. This proves the existence of $y$. To prove its uniqueness, note that
$$
\Psi_{\wh y}(R[q])\subset \Psi_{\wh f^{-1}(\wh x)}(R[e^{2\ve} q])\subset
\Psi_{\wh f^{-1}(\wh x)}(R[10Q_\ve(\wh f^{-1}(\wh x))])\subset g_{\wh x}(E_{\vt_{-1}[\wh x]}),
$$
where in the first inclusion we used Proposition \ref{Lemma-overlap}(2) and in the last
we used the third item proved before Lemma \ref{Lemma-temperedness}. Since
$g_{\wh x}:E_{\vt_{-1}[\wh x]}\to g_{\wh x}(E_{\vt_{-1}[\wh x]})$ is a diffeomorphism, there
is at most one $y\in g_{\wh x}(E_{\vt_{-1}[\wh x]})$ s.t. $f(y)=x$.
\end{proof}

\medskip
\noindent
{\sc $\ve$--generalized pseudo-orbit ($\ve$--gpo):} An {\em $\ve$--generalized pseudo-orbit {\rm (}$\ve$--gpo{\rm )}}
is a sequence $\un{v}=\{v_n\}_{n\in\Z}$ of $\ve$--charts
s.t. $v_n\overset{\ve}{\leftarrow}v_{n+1}$ for all $n\in\Z$.
A {\em weak $\ve$--gpo} is a sequence $\un{v}=\{v_n\}_{n\in\Z}$ of $\ve$--charts
s.t. $v_n\overset{\ve}{\dashleftarrow}v_{n+1}$ for all $n\in\Z$.

\medskip
By definition, a necessary condition for drawing a weak edge
$\Psi_{\wh y}^q\overset{\ve}{\dashleftarrow}\Psi_{\wh x}^p$
is that $p\leq e^{\ve}q$. We would like to draw an edge
$\Psi_{\wh x}^{Q_\ve(\wh x)}\overset{\ve}{\dashleftarrow}\Psi_{\wh f(\wh x)}^{Q_\ve(\wh f(\wh x))}$,
but in general we cannot, because $\tfrac{Q_\ve(\wh f(\wh x))}{Q_\ve(\wh x)}$ 
might be bigger than $e^{\ve}$. To bypass this, we introduce the parameter $q_\ve(\wh x)$ below.

\medskip
\noindent
{\sc Parameter $q_\ve(\wh x)$:} For $\wh x\in{\rm NUE}_\chi^*$, let
$q_\ve(\wh x):=\delta_\ve\min\{e^{\ve|n|}Q_\ve(\wh f^n(\wh x)):n\leq 0\}$.

\medskip
The above minimum is the greedy way of defining values in $I_\ve$ smaller than $\ve Q_\ve$
with the required regularity property, as we now show.

\begin{lemma}\label{Lemma-q}
For all $\wh x\in{\rm NUE}_\chi^*$, the following holds:
\begin{enumerate}[{\rm (1)}]
\item {\sc Good definition:} $0<q_\ve(\wh x)<\ve Q_\ve(\wh x)$.
\item {\sc Greedy algorithm:}
$q_\ve(\wh f^n(\wh x))=\min\{e^\ve q_\ve(\wh f^{n-1}(\wh x)),\delta_\ve Q_\ve(\wh f^n(\wh x))\}$,
$\forall n\in\Z$.
\end{enumerate}
\end{lemma}

\begin{proof}
By Lemma \ref{Lemma-temperedness}, $\inf\{e^{\ve|n|}Q_\ve(\wh f^n(\wh x)):n\leq 0\}>0$.
Since zero is the only accumulation point of $I_\ve$,
$q_\ve(\wh x)$ is well-defined and positive.
It is clear that $q_\ve(\wh x)\leq \delta_\ve Q_\ve(\wh x)<\ve Q_\ve(\wh x)$,
hence (1) is proved. For (2), fix $n\in\Z$ and note that
\begin{align*}
&\ q_\ve(f^n(\wh x))=\delta_\ve\min\{e^{\ve|m|}Q_\ve(\wh f^m(\wh f^n(\wh x))):m\leq 0\}\\
&=\min\{\delta_\ve \min\{e^{\ve|m|}Q_\ve(\wh f^{m+n}(\wh x)):m\leq -1\},
\delta_\ve Q_\ve(\wh f^n(\wh x))\}\\
&=\min\{e^\ve\delta_\ve\min\{e^{\ve|m|}Q_\ve(\wh f^m(\wh f^{n-1}(x))):m\leq 0\},
\delta_\ve Q_\ve(\wh f^n(\wh x))\}\\
&=\min\{e^\ve q_\ve(\wh f^{n-1}(\wh x)),\delta_\ve Q_\ve(\wh f^n(\wh x))\}.
\end{align*}
\end{proof}

\medskip
\noindent
{\sc The set ${\rm NUE}_\chi^\#$:} It is the set of $\wh x\in{\rm NUE}_\chi^*$ s.t.
$\limsup_{n\to+\infty}q_\ve(\wh f^n(\wh x))>0$ and $\limsup_{n\to-\infty}q_\ve(\wh f^n(\wh x))>0$.

\medskip
Note that, while ${\rm NUE}_\chi^*$ is defined by a set of conditions on the past orbit of $\wh x$,
the set ${\rm NUE}_\chi^\#$ is defined by conditions {\em both} on the past and on the future.
This additional condition is important for the proof of Theorem \ref{Thm-widehat-pi}.

\subsection{Stable and unstable sets of weak $\ve$--gpo's}\label{Section-stable-sets}

Call a sequence ${\un v}^+=\{v_n\}_{n\geq 0}$ a {\em positive weak $\ve$--gpo} if
$v_n\overset{\ve}{\dashleftarrow}v_{n+1}$ for all $n\geq 0$.
Similarly, a {\em negative weak $\ve$--gpo} is a sequence ${\un v}^-=\{v_n\}_{n\leq 0}$
s.t. $v_{n-1}\overset{\ve}{\dashleftarrow}v_n$ for all $n\leq 0$.
Remember $\vt_n:\wh M\to M$, $\vt_n[(x_k)_{k\in\Z}]=x_n$.

\medskip
\noindent
{\sc Stable/unstable set of positive/negative weak $\ve$--gpo:} The {\em stable set}
of a positive weak $\ve$--gpo ${\un v}^+=\{\Psi_{\wh x_n}^{p_n}\}_{n\geq 0}$ is 
$$
V^s[{\un v}^+]:=\{\wh x\in\wh M:\vt_n[\wh x]\in\Psi_{\wh x_n}(R[p_n]),\forall n\geq 0\}.
$$
The {\em unstable set} of a negative weak $\ve$--gpo ${\un v}^-=\{\Psi_{\wh x_n}^{p_n}\}_{n\leq 0}$ is 
$$
V^u[{\un v}^-]:=\{\wh x\in\wh M:\vt_n[\wh x]\in\Psi_{\wh x_n}(R[p_n]),\forall n\leq 0\}.
$$
For a weak $\ve$--gpo $\un{v}=\{v_n\}_{n\in\Z}$, let $V^s[\un v]:=V^s[\{v_n\}_{n\geq 0}]$
and $V^u[\un v]:=V^u[\{v_n\}_{n\leq 0}]$.

\medskip
\noindent
{\sc Stable/unstable sets at $v$:} Given an $\ve$--chart $v$, a {\em stable set at $v$}
is any $V^s[\un v^+]$ where $\un v^+$ is a positive weak $\ve$--gpo with $v_0=v$.
Similarly, an {\em unstable set at $v$} is any $V^u[\un v^-]$ where $\un v^-$ is a negative
weak $\ve$--gpo with $v_0=v$. 

\medskip
In the sequel, the notations $\un v^+,\{v_n\}_{n\geq 0}$ always mean a positive weak $\ve$--gpo,
and the notations $\un v^-,\{v_n\}_{n\geq 0}$ always mean a negative weak $\ve$--gpo.
The next lemma gives alternative characterizations of stable and unstable sets.

\begin{lemma}\label{Lemma-stable-sets}
The following holds for all $\ve>0$ small enough.
\begin{enumerate}[{\rm (1)}]
\item If $\un v^+=\{\Psi_{\wh x_n}^{p_n}\}_{n\geq 0}$ is a positive weak $\ve$--gpo
then $V^s[\un v^+]=\vt^{-1}[x]$, where $x\in M$ is uniquely
defined by $f^n(x)\in \Psi_{\wh x_n}(R[p_n])$ for all $n\geq 0$.
\item If $\un v^-=\{\Psi_{\wh x_n}^{p_n}\}_{n\leq 0}$ is a negative weak $\ve$--gpo
then 
\begin{align*}
V^u[\un v^-]&=\{\wh x=(x_n)_{n\in\Z}\in\wh M:x_0\in \Psi_{\wh x_0}(R[p_0])\text{ and }
x_{n-1}=g_{\wh x_n}(x_n),\forall n\leq 0 \}\\
&=\{(\Psi_{\wh x_n}(t_n))_{n\in\Z}\in\wh M:t_0\in R[p_0]\text{ and }
t_{n-1}=G_{\wh x_n,\wh x_{n-1}}(t_n),\forall n\leq 0\}.
\end{align*}
\end{enumerate}
\end{lemma}

In other words, a stable set is the set of all possible pasts of a single $x\in M$,
and an unstable set is isomorphic to the interval $\Psi_{\wh x_0}(R[p_0])$, i.e. 
an element of an unstable set is uniquely determined by its zeroth coordinate.

\begin{proof} Let $\un v^+=\{\Psi_{\wh x_n}^{p_n}\}_{n\geq 0}$ be a positive weak $\ve$--gpo.
Firstly we prove that there exists a unique $x\in M$ s.t.
$f^n(x)\in \Psi_{\wh x_n}(R[p_n])$, $\forall n\geq 0$. Any such $x$ is defined by a sequence 
$(t_n)_{n\geq 0}$ of values $t_n\in R[p_n]$ s.t. $f^n(x)=\Psi_{\wh x_n}(t_n)$ and
$t_n=G_{\wh x_{n+1},\wh x_n}(t_{n+1})$ for all $n\geq 0$.
By Theorem \ref{Thm-non-linear-Pesin-2},
each $G_{\wh x_{n+1},\wh x_n}$ contracts at least by a factor $e^{-\chi/2}$,
hence $t_0$ is the intersection of the descending chain of compact intervals
$I_n:=(G_{\wh x_1,\wh x_0}\circ \cdots\circ G_{\wh x_n,\wh x_{n-1}})(R[p_n])$,
$n\geq 0$. By a similar reasoning, $t_n$ is uniquely defined for all $n\geq 0$.
By this uniqueness, $t_n=G_{\wh x_{n+1},\wh x_n}(t_{n+1})$ for all $n\geq 0$.

\medskip
\noindent
(1) Let $x\in M$ s.t. $f^n(x)\in \Psi_{\wh x_n}(R[p_n])$, $\forall n\geq 0$.
Take $\wh x\in V^s[\un v^+]$, and let $x_0=\vartheta[\wh x]$. Since
$f^n(x_0)=\vt_n[\wh x]\in \Psi_{\wh x_n}(R[p_n])$ for all $n\geq 0$,
we have $x_0=x$, thus $V^s[\un v^+]\subset \vt^{-1}[x]$. Conversely,
$\wh y\in\vt^{-1}[x]\Rightarrow \vt_n[\wh y]=f^n(x)\in \Psi_{\wh x_n}(R[p_n])$ for all $n\geq 0$,
hence $V^s[\un v^+]\supset\vt^{-1}[x]$.

%

\medskip
\noindent
(2) Fix a negative weak $\ve$--gpo $\un v^-=\{\Psi_{\wh x_n}^{p_n}\}_{n\leq 0}$.
It is easy to see that the two alternative characterizations of $V^u[\un v^-]$
are equivalent: $x_0\in\Psi_{\wh x_0}(R[p_0])$ iff $x_0=\Psi_{\wh x_0}(t_0)$
for some $t_0\in R[p_0]$; $x_{n-1}=g_{\wh x_n}(x_n)$ iff $x_{n-1}=\Psi_{\wh x_{n-1}}(t_{n-1})$
and $x_n=\Psi_{\wh x_n}(t_n)$ with $t_{n-1}=G_{\wh x_n,\wh x_{n-1}}(t_n)$.
Hence it is enough to show the second characterization.
Take $\wh x=(x_n)_{n\in\Z}\in V^u[\un v^-]$. Fix $n\leq 0$. By assumption,
$x_{n-1}\in \Psi_{\wh x_{n-1}}(R[p_{n-1}])$, $x_n\in \Psi_{\wh x_n}(R[p_n])$ and
$f(x_{n-1})=x_n$. By Lemma \ref{Lemma-edge}(2), it follows that $x_{n-1}=g_{\wh x_n}(x_n)$.

\medskip
Conversely, take $\wh x=(x_n)_{n\in\Z}\in\wh M$ s.t. $x_0\in \Psi_{\wh x_0}(R[p_0])$
and $x_{n-1}=g_{\wh x_n}(x_n)$ for all $n\leq 0$.
By Lemma \ref{Lemma-edge}(1), we have
$g_{\wh x_n}(\Psi_{\wh x_n}(R[p_n]))\subset \Psi_{\wh x_{n-1}}(R[p_{n-1}])$ for all $n\in\Z$.
Applying this for $n=0$, we get that $x_{-1}=g_{\wh x_0}(x_0)\in \Psi_{\wh x_{-1}}(R[p_{-1}])$.
By induction, it follows that $x_n\in \Psi_{\wh x_n}(R[p_n])$ for all $n\leq 0$.
\end{proof}

Here are the main properties of stable and unstable sets.

\begin{proposition}\label{Prop-stable-manifolds}
The following holds for all $\ve>0$ small enough.
\begin{enumerate}[{\rm (1)}]
\item {\sc Product structure:} If $V^s/V^u$ is a stable$/$unstable set at $v$
then $V^s\cap V^u$ consists of  a single element of $\wh M$.
\item {\sc Invariance:}
$$
\wh f(V^s[\{v_n\}_{n\geq 0}])\subset V^s[\{v_n\}_{n\geq 1}]\text{ and }
\wh f^{-1}(V^u[\{v_n\}_{n\leq 0}])\subset V^u[\{v_n\}_{n\leq -1}].
$$
\item {\sc Hyperbolicity:} If $\wh y,\wh z\in V^s[{\un v}^+]$ then
$d(\wh f^n(\wh y),\wh f^n(\wh z))=2^{-n}d(\wh y,\wh z)$ for all $n\geq 0$.
If $\wh y,\wh z\in V^u[\{\Psi_{\wh x_n}^{p_n}\}_{n\leq 0}]$
then for all $n\leq 0$:
\begin{enumerate}[{\rm (a)}]
\item  $d(\wh f^n(\wh y),\wh f^n(\wh z))\leq 2p_0e^{\frac{\chi}{2}n}d(\wh y,\wh z)$.
\item $|\log\|\wh{df}^{(n)}_{\wh y}\|-\log\|\wh{df}^{(n)}_{\wh z}\||<Q_\ve(\wh x_0)^{\beta/4}$.
In particular, $\tfrac{u(\wh y)}{u(\wh z)}=e^{\pm Q_\ve(\wh x_0)^{\beta/4}}$.
\end{enumerate}
\item {\sc Disjointness:} Let $v=\Psi_{\wh x}^p$ and $w=\Psi_{\wh y}^q$ with $\wh x=\wh y$.
If $V^s,W^s$ are stable sets at $v,w$ then they are either disjoint or coincide. 
If $V^u,W^u$ are unstable sets at $v,w$ then they are either disjoint or one contains the other.
\item Let $v\overset{\ve}{\dashleftarrow}w$. If $V^u$ is an unstable set at
$v$ then $\wh f(V^u)$ intersects every stable set at $w$ at a single element.
\end{enumerate}
\end{proposition}

\begin{proof}
(1) Write $v=\Psi_{\wh x_0}^{p_0}$, $V^s=V^s[\{\Psi_{\wh x_n}^{p_n}\}_{n\geq 0}]$,
$V^u=V^u[\{\Psi_{\wh x_n}^{p_n}\}_{n\leq 0}]$. By Lemma \ref{Lemma-stable-sets}(1),
$\exists x\in M$ s.t. $V^s=\vt^{-1}[x]$. Any element $\wh x=(x_n)_{n\in\Z}$ in 
$V^s\cap V^u$ satisfies $x_n=f^n(x)$ for all $n\geq 0$, and $x_{n-1}=g_{\wh x_n}(x_n)$ for all $n\leq 0$.
These conditions uniquely characterize $\wh x$, hence $V^s\cap V^u$ is a singleton.

\medskip
\noindent
(2) If $\un v^+=\{\Psi_{\wh x_n}^{p_n}\}_{n\geq 0}$ is a positive weak $\ve$--gpo then
$\wh x\in V^s[\un v^+]\Rightarrow \vt_n[\wh x]\in\Psi_{\wh x_n}(R[p_n])$, $\forall n\geq 0\Rightarrow
\vt_n[\wh f(\wh x)]\in \Psi_{\wh x_{n+1}}(R[p_{n+1}]),\forall n\geq 0\Rightarrow \wh f(\wh x)\in V^s[\{v_n\}_{n\geq 1}]$.
By a similar reason, if $\un v^-=\{\Psi_{\wh x_n}^{p_n}\}_{n\leq 0}$ is a negative weak $\ve$--gpo then
$\wh x\in V^u[\un v^-]\Rightarrow \vt_n[\wh x]\in\Psi_{\wh x_n}(R[p_n])$, $\forall n\leq 0\Rightarrow
\vt_n[\wh f^{-1}(\wh x)]\in \Psi_{\wh x_{n-1}}(R[p_{n-1}]),\forall n\leq 0\Rightarrow
\wh f^{-1}(\wh x)\in V^u[\{v_n\}_{n\leq -1}]$.
%

\medskip
\noindent
(3) Write $V^s[\un v^+]=\vt^{-1}(x)$, and let $\wh y=(y_n)_{n\in\Z},\wh z=(z_n)_{n\in\Z}$ be in $V^s[\un v^+]$.
For $n\geq 0$ we have $y_n=z_n=f^n(x)$, thus
$d(\wh f^n(\wh y),\wh f^n(\wh z))=\sup\{2^{-k}d(y_{n-k},z_{n-k}):k\geq 0\}=\sup\{2^{-k}d(y_{n-k},z_{n-k}):k\geq n\}
=2^{-n}d(\wh y,\wh z)$. 

\medskip
Now let $\un v^-=\{\Psi_{\wh x_n}^{p_n}\}_{n\leq 0}$ be a negative weak $\ve$--gpo, and
take $\wh y=(y_n)_{n\in\Z},\wh z=(z_n)_{n\in\Z}\in V^u[\un v^-]$.
By Lemma \ref{Lemma-stable-sets}(2), for all $n\leq 0$ we can write
$y_n=\Psi_{\wh x_n}(t_n)$, $z_n=\Psi_{\wh x_n}(t_n')$, where $t_0,t_0'\in R[p_0]$
and $t_{n-1}=G_{\wh x_n,\wh x_{n-1}}(t_n)$, $t_{n-1}'=G_{\wh x_n,\wh x_{n-1}}(t_n')$.
Define $\Delta_n:=t_n-t_n'$ for $n\leq 0$. By Theorem \ref{Thm-non-linear-Pesin-2},
$|\Delta_{n-1}|\leq e^{-\frac{\chi}{2}}|\Delta_n|$ for all $n\leq 0$, therefore 
$|\Delta_n|\leq e^{\frac{\chi}{2}n}|\Delta_0|\leq 2p_0e^{\frac{\chi}{2}n}$ for all $n\leq 0$,
and so $d(y_n,z_n)\leq 2p_0e^{\frac{\chi}{2}n}$ (since $\Psi_{\wh x_n}$ is 1--Lipschitz).
We conclude that $d(\wh f^n(\wh y),\wh f^n(\wh z))\leq 2p_0e^{\frac{\chi}{2}n}d(\wh y,\wh z)$ for all $n\leq 0$.
To prove (b), we proceed exactly as in the proof of Proposition 6.2(1)(c) of \cite{Lima-Matheus}.

\medskip
\noindent
(4) Let $V^s,W^s$ be stable sets in $v,w$ respectively. By Lemma \ref{Lemma-stable-sets}(1),
$\exists y,z\in M$ s.t. $V^s=\vt^{-1}[y]$ and $W^s=\vt^{-1}[z]$, hence either
$V^s\cap W^s=\emptyset$ or $V^s=W^s$.

\medskip
Now let $V^u=V^u[\{\Psi_{\wh x_n}^{p_n}\}_{n\leq 0}]$
and $W^u=V^u[\{\Psi_{\wh y_n}^{q_n}\}_{n\leq 0}]$, with $v=\Psi_{\wh x}^{p}$,
$w=\Psi_{\wh x}^{q}$. If $V^u\cap W^u=\emptyset$ then there is nothing to prove, so 
assume that $V^u\cap W^u\neq\emptyset$. Assuming that $p\leq q$, we will prove
that $V^u\subset W^u$ (the other case is identical). In the sequel, ``$n$ small enough''
means that $n\leq 0$ and $|n|$ is large enough. The following claims hold.
\begin{enumerate}[$\circ$]
\item If $n$ is small enough then $\vt_n[V^u]\subset \Psi_{\wh x_n}(R[\tfrac{1}{2}p_n])$:
take $\wh x=(x_n)_{n\in\Z}\in V^u$. For $n\leq 0$ write $x_n=\Psi_{\wh x_n}(t_n)$ with
$t_0\in R[p_0]$ and $t_{n-1}=G_{\wh x_n,\wh x_{n-1}}(t_n)$, where 
$G_{\wh x_n,\wh x_{n-1}}(t)=A_nt+h_n(t)$. It is enough to show that $|t_n|<\tfrac{1}{2}p_n$
for $n$ small enough. Start noting that $\|dh_n\|_0<2\ve^2$, where the
norm is taken in $R[\tfrac{1}{2}p_n]$. This is a direct consequence of Theorem \ref{Thm-non-linear-Pesin-2}:
$\|dh_n\|_0\leq |d(h_n)_0|+\ve p_n^{\beta/3}<\ve p_n^6+\ve^2<2\ve^2$.
Theorem \ref{Thm-non-linear-Pesin-2} also says that
$|A_n|<e^{-\chi}$ and $|h_n(0)|< \ve p_n^6$, therefore if $\ve>0$ is small enough then the following
holds for all $n\leq 0$:
\begin{align*}
|t_{n-1}|\leq |A_n||t_n|+|h_n(t_n)|< e^{-\chi}|t_n|+\ve p_n^6+2\ve^2 |t_n|<e^{-\frac{\chi}{2}}|t_n|+\ve p_n.
\end{align*}
By (WE2), $p_k\leq e^{\ve(k-\ell)}p_\ell$ whenever $\ell\leq k$, hence for all $n\leq 0$ we have
\begin{align*}
&\ |t_n|\leq e^{\frac{\chi}{2}n}|t_0|+\ve(p_{n+1}+e^{-\frac{\chi}{2}}p_{n+2}+\cdots+e^{\frac{\chi}{2}(n+1)}p_0)\\
&\leq e^{\frac{\chi}{2}n}p_0+\ve e^{\ve}(p_n+e^{-\frac{\chi}{2}}p_{n+1}+\cdots+e^{\frac{\chi}{2}(n+1)}p_{-1})\\
&\leq \left[e^{(\frac{\chi}{2}-\ve)n}+\ve e^{\ve}\sum_{i=n+1}^0 e^{(\frac{\chi}{2}-\ve)i}\right]p_n<\tfrac{1}{2}p_n,\\
\end{align*}
since $e^{(\frac{\chi}{2}-\ve)n}<\tfrac{1}{4}$ for $n$ small enough and
$\ve e^{\ve}\sum_{i=n+1}^0 e^{(\frac{\chi}{2}-\ve)i}<\tfrac{2\ve}{1-e^{-\chi/4}}<\tfrac{1}{4}$.
\item If $n$ is small  enough then $\vt_n[W^u]\subset \Psi_{\wh x_n}(R[p_n])$: 
fix $\wh y=(y_n)_{n\in\Z}\in V^u\cap W^u$, and take $\wh z=(z_n)_{n\in\Z}\in W^u$.
In part (3) we proved that $d(y_n,z_n)\leq 2q_0e^{\frac{\chi}{2}n}$ for all $n\leq 0$,
thus
$$
|\Psi_{\wh x_n}^{-1}(y_n)-\Psi_{\wh x_n}^{-1}(z_n)|\leq 2q_0u(\wh x_n) e^{\frac{\chi}{2}n}\leq
2q_0p_n^{-1} e^{\frac{\chi}{2}n}\leq 2q_0p_0^{-1}e^{(\frac{\chi}{2}-\ve)n},
$$
since $u(\wh x_n)\leq Q_\ve(\wh x_n)^{-1}\leq p_n^{-1}$ and 
$p_0\leq e^{-\ve n}p_n$. If $n$ is small enough then 
$\Psi_{\wh x_n}^{-1}(y_n)\in R[\tfrac{1}{2}p_n]$ and
$2q_0p_0^{-1}e^{(\frac{\chi}{2}-\ve)n}<\tfrac{1}{2}p_0 e^{\ve n}\leq \tfrac{1}{2}p_n$,
hence $\Psi_{\wh x_n}^{-1}(z_n)\in R[p_n]$.
\end{enumerate}
Fix $n\leq 0$ s.t. both items above hold for all $N\leq n$. By the definition of unstable sets,
we have $\wh f^n(V^u),\wh f^n(W^u)\subset V^u[\{\Psi_{\wh x_k}^{p_k}\}_{k\leq n}]$. 
By Lemma \ref{Lemma-stable-sets}(2),
and since the inverse branches $g_{\wh x_k}$ send intervals onto intervals,
$\exists\alpha,\beta,\alpha',\beta'\in\R$ s.t.:
\begin{align*}
\wh f^n(V^u)&=\left\{(x_k)_{k\in\Z}\in\wh M:
\begin{array}{l}
\text{for }k\leq 0\text{ we can write }x_k=\Psi_{\wh x_{n+k}}(t_k)\text{ with}\\
t_0\in[\alpha,\beta]\text{ and }t_{k-1}=G_{\wh x_{n+k},\wh x_{n+k-1}}(t_k)
\end{array}\right\}\\
&\\
\wh f^n(W^u)&=\left\{(x_k)_{k\in\Z}\in\wh M:
\begin{array}{l}
\text{for }k\leq 0\text{ we can write }x_k=\Psi_{\wh x_{n+k}}(t_k')\text{ with}\\
t_0'\in[\alpha',\beta']\text{ and }t_{k-1}'=G_{\wh x_{n+k},\wh x_{n+k-1}}(t_k')
\end{array}\right\}.
\end{align*}
To prove that $V^u\subset W^u$ it is enough to show that
$[\alpha,\beta]\subset[\alpha',\beta']$: if this happens then
$\wh f^n(V^u)\subset \wh f^n(W^u)$ and thus $V^u\subset W^u$. By contradiction,
assume that $[\alpha,\beta]\not\subset[\alpha',\beta']$, then either $\alpha<\alpha'$ and/or
$\beta'<\beta$. By symmetry, we may assume $\alpha<\alpha'$. Thus $A'=\Psi_{\wh x_n}(\alpha')$
belongs to the interior of the interval with endpoints $A=\Psi_{\wh x_n}(\alpha)$ and
$B=\Psi_{\wh x_n}(\beta)$. Since $f$ is continuous inside the ranges of $\ve$--charts,
$f^{-n}(A')$ belongs to the interior of the interval with endpoints
$f^{-n}(A)$ and $f^{-n}(B)$. This latter interval is $\Psi_{\wh x}(R[p])$. But
$f^{-n}(A')$ is one of the endpoints of $\Psi_{\wh x}(R[q])$, therefore 
$f^{-n}(A')\in \Psi_{\wh x}(R[p])$ iff $q<p$, which contradicts our assumption.
The proof is complete.

\medskip
\noindent
(5) Write $v=\Psi_{\wh x_0}^{p_0}$, $w=\Psi_{\wh x_1}^{p_1}$ and
$V^u=V^u[\un v^-]$ where $\un v^-=\{\Psi_{\wh x_n}^{p_n}\}_{n\leq 0}$ is a
negative weak $\ve$--gpo. Let $V^s=\vt^{-1}[x]$ be a stable set at $w$. We want to show
that $\wh f(V^u)\cap V^s$ consists of a single element. By Lemma \ref{Lemma-stable-sets}(2),
$$
V^u=\{\wh x=(x_n)_{n\in\Z}\in\wh M:x_0\in \Psi_{\wh x_0}(R[p_0])\text{ and }
x_{n-1}=g_{\wh x_n}(x_n),\forall n\leq 0 \}
$$
hence 
$$
\wh f(V^u)=\{\wh x=(x_n)_{n\in\Z}\in\wh M:x_{-1}\in\Psi_{\wh x_0}(R[p_0])\text{ and }
x_{n-1}=g_{\wh x_n}(x_n),\forall n\leq -1\}.
$$
Any $\wh x=(x_n)_{n\in\Z}\in \wh f(V^u)\cap V^s$ must satisfy
$x_n=f^n(x)$ for $n\geq 0$ and $x_{n-1}=g_{\wh x_n}(x_n)$ for $n\leq -1$,
therefore $\wh x$ is uniquely defined by the choice of $x_{-1}$. By Lemma
\ref{Lemma-edge}(2), there is a unique $x_{-1}\in\Psi_{\wh x_0}(R[p_0])$
s.t. $f(x_{-1})=x$.
\end{proof}

\medskip
\noindent
{\sc Shadowing:} A weak $\ve$--gpo $\{\Psi_{\wh x_n}^{p_n}\}_{n\in\Z}$ is said to {\em shadow}
a point $\wh x\in\wh M$ if $\vt_n[\wh x]\in \Psi_{\wh x_n}(R[p_n])$ for all $n\in\Z$.

\begin{lemma}\label{Lemma-shadowing}
Every weak $\ve$--gpo shadows a unique element of $\wh M$.
\end{lemma}

\begin{proof}
Let $\un v=\{v_n\}_{n\in\Z}$ be a weak $\ve$--gpo, and let $V^s=V^s[\{v_n\}_{n\geq 0}]$,
$V^u=V^u[\{v_n\}_{n\leq 0}]$. By the definition of $V^s$ and $V^u$, 
any point $\wh x\in\wh M$ shadowed by $\un v$ belongs to $V^s\cap V^u$.
By Proposition \ref{Prop-stable-manifolds}(1), this intersection consists of a single element of $\wh M$.
\end{proof}

\section{Coarse graining}\label{Section-coarse-graining}

In this section, we construct a countable set of $\ve$--charts whose set of (strong) $\ve$--gpo's
shadows all relevant orbits of $\wh f$.

\begin{theorem}\label{Thm-coarse-graining}
For all $\ve>0$ sufficiently small, there exists a countable family $\mathfs A$ of $\ve$--charts
with the following properties:
\begin{enumerate}[{\rm (1)}]
\item {\sc Discreteness}: For all $t>0$, the set $\{\Psi_{\wh x}^p\in\mathfs A:p>t\}$ is finite.
\item {\sc Sufficiency:} If $\wh x\in {\rm NUE}_\chi^*$ then there is an $\ve$--gpo
$\un v\in{\mathfs A}^{\Z}$ that shadows $\wh x$.
\item {\sc Relevance:} For all $v\in \mathfs A$ there is an $\ve$--gpo $\un{v}\in\mathfs A^\Z$
with $v_0=v$ that shadows a point in ${\rm NUE}_\chi^*$.
\end{enumerate}
\end{theorem}

Parts (1) and (3) are essential to prove the inverse theorem (Theorem \ref{Thm-inverse}).
Part (2) and Lemma \ref{Lemma-adaptedness} imply that if $\mu$ is $f$--adapted 
and $\chi$--expanding then $\wh\mu$--a.e. $\wh x\in\wh M$ is shadowed
by an $\ve$--gpo whose vertices belong to $\mathfs A$.


\begin{proof}
When $M$ is compact and $f$ is a diffeomorphism, the above statement is consequence
of Propositions 3.5, 4.5 and Lemmas 4.6, 4.7 of \cite{Sarig-JAMS}. When $M$ is compact with boundary
and $f$ is a local diffeomorphism with bounded derivatives, this is \cite[Prop. 4.3]{Lima-Sarig}.
When $f$ is a surface map with discontinuities and possibly unbounded derivatives,
this is \cite[Theorem 5.1]{Lima-Matheus}.
We follow the strategy of \cite{Lima-Matheus}, adapted to our context.

\medskip
For $t>0$, let $M_t=\{x\in M: d(x,\mathfs S)\geq t\}$.
Since $M$ has finite diameter (we are even assuming it is smaller than one),
each $M_t$ is compact.
Let $\N_0=\N\cup\{0\}$. Fix a countable open cover $\mathfs P=\{D_i\}_{i\in\N_0}$ of $M\backslash\mathfs S$ s.t.:
\begin{enumerate}[$\circ$]
\item $D_i:=D_{z_i}=B(z_i,2\mathfrak r(z_i))$ for some $z_i\in M$.
\item For every $t>0$, $\{D\in\mathfs P:D\cap M_t\neq\emptyset\}$ is finite.
\end{enumerate}

\medskip
Let $X:=M^3\times (0,\infty)^3\times (0,1]$. For $\wh x\in{\rm NUE}_\chi^*$, let
$\Gamma(\wh x)=(\un{\wh x},\un u,\un Q)\in X$ where:
\begin{align*}
&\un{\wh x}=(\vt_{-1}[\wh x],\vt_0[\wh x],\vt_1[\wh x]),
\ \un u=(u(\wh f^{-1}(\wh x)),u(\wh x),u(\wh f(\wh x))),\ \un Q=Q_\ve(\wh x).
\end{align*}
Let $Y=\{\Gamma(\wh x):\wh x\in{\rm NUE}_\chi^*\}$. We want to construct a countable dense subset
of $Y$. Since the maps $\wh x\mapsto u(\wh x),Q_\ve(\wh x)$ are not necessarily continuous,
we apply a precompactness argument.
For vectors $\un{k}=(k_{-1},k_0,k_1),\un{\ell}=(\ell_{-1},\ell_0,\ell_1),\un a=(a_{-1},a_0,a_1)\in\N_0^3$
and $m\in\N_0$, define
$$
Y_{\un k,\un \ell,\un a,m}:=\left\{\Gamma(\wh x)\in Y:
\begin{array}{cl}
e^{-k_i-1}\leq d(\vt_i[\wh x],\mathfs S)< e^{-k_i},& -1\leq i\leq 1\\
e^{\ell_i}\leq u(\wh f^i(\wh x))<e^{\ell_i+1},&-1\leq i\leq 1\\
\vt_i[\wh x]\in D_{a_i},&-1\leq i\leq 1\\
e^{-m-1}\leq Q_\ve(\wh x)<e^{-m}&\\
\end{array}
\right\}.
$$

\medskip
\noindent
{\sc Claim 1:} $Y=\bigcup_{\un k,\un\ell,\un a\in\N_0^3\atop{m\in\N_0}}Y_{\un k,\un\ell,\un a,m}$, and each
$Y_{\un k,\un\ell,\un a,m}$ is precompact in $X$.

\medskip
\noindent
{\em Proof of Claim $1$.}
The first statement is clear, so we focus on the second.
Fix $\un k,\un\ell,\un a\in \N_0^3$, $m\in\N_0$, and
take $\Gamma(\wh x)\in Y_{\un k,\un\ell,\un a,m}$. Then
$$
\un{\wh x}\in M_{e^{-k_{-1}-1}}\times M_{e^{-k_0-1}}\times M_{e^{-k_1-1}},$$
a precompact subset of $M^3$.
For $|i|\leq 1$ we have $1\leq u(\wh f^i(\wh x))< e^{\ell_i+1}$, hence $\un u$ belongs to a compact
subset of $(0,\infty)^3$. Also $Q_\ve(\wh x)\in [e^{-m-1},1]$, therefore $\un Q$ belongs to a compact
subinterval of $(0,1]$. The product of precompact sets is precompact, thus the claim is proved.

\medskip
Let $j\geq 0$. By Claim 1, there is a finite set
$Y_{\un k,\un\ell,\un a,m}(j)\subset Y_{\un k,\un\ell,\un a,m}$
s.t. for every $\Gamma(\wh x)\in Y_{\un k,\un\ell,\un a,m}$
there exists $\Gamma(\wh y)\in Y_{\un k,\un\ell,\un a,m}(j)$
s.t.:
\begin{enumerate}[{\rm (a)}]
\item $d(\vt_i[\wh x],\vt_i[\wh y])+
|u(\wh f^i(\wh x))^{-1}-u(\wh f^i(\wh y))^{-1}|<e^{-8(j+2)}$
for $|i|\leq 1$.
\item $\tfrac{Q_\ve(\wh x)}{Q_\ve(\wh y)}=e^{\pm\frac{\ve}{3}}$.
\end{enumerate}
Remind that $I_\ve:=\{e^{-\frac{1}{3}\ve n}:n\geq 0\}$.

\medskip
\noindent
{\sc The alphabet $\mathfs A$:} Let $\mathfs A$ be the countable family of $\Psi_{\wh x}^p$ s.t.:
\begin{enumerate}[i i)]
\item[(CG1)] $\Gamma(\wh x)\in Y_{\un k,\un\ell,\un a,m}(j)$ for some
$(\un k,\un\ell,\un a,m,j)\in\N_0^3\times\N_0^3\times\N_0^3\times\N_0\times\N_0$.
\item[(CG2)] $p\in I_\ve$, $p\leq \delta_\ve Q_\ve(\wh x)$, and $e^{-j-2}\leq p\leq e^{-j+2}$.
\end{enumerate}

\medskip
\noindent
{\em Proof of discreteness.}
By the proof of Lemma \ref{Lemma-linear-reduction} and assumption (A2),
$$
u(\wh f(\wh x))^2=1+e^{2\chi}|df_{\vt_0[\wh x]}|^{-2}u(\wh x)^2\leq 2e^{2\chi}\rho(\wh x)^{-2a}u(\wh x)^2
<e^{2\chi+2}\rho(\wh x)^{-2a}u(\wh x)^2
$$
hence $u(\wh f(\wh x))<e^{\chi+1}\rho(\wh x)^{-a}u(\wh x)$. We will use this estimate below.

\medskip
Fix $0<t<1$, and let $\Psi_{\wh x}^p\in\mathfs A$ with $p>t$.
Start noting that $\rho(\wh x)>\rho(\wh x)^{2a}>Q_\ve(\wh x)>p>t$.
If $\Gamma(\wh x)\in Y_{\un k,\un\ell,\un a,m}(j)$ then:
\begin{enumerate}[$\circ$]
\item Finiteness of $\un k$: for $|i|\leq 1$,
$e^{-k_i}> d(\vt_i[\wh x],\mathfs S)\geq\rho(\wh x)>t$, hence $k_i< |\log t|$.
\item Finiteness of $\un\ell$: for $i=-1,0$, $e^{\ell_i}\leq u(\wh f^i(\wh x))<Q_\ve(\wh x)^{-1}<t^{-1}$,
hence $\ell_i<|\log t|$. For $i=1$, the estimate in the beginning of the proof implies
$e^{\ell_1}\leq u(\wh f(\wh x))<e^{\chi+1}\rho(\wh x)^{-a}u(\wh x)<e^{\chi+1}t^{-(a+1)}$,
hence $\ell_1<\chi+1+(a+1)|\log t|=:T_t$, which is bigger than $|\log t|$.
\item Finiteness of $\un a$: for $|i|\leq 1$, $\vt_i[\wh x]\in D_{a_i}\cap M_t$ hence $D_{a_i}$ 
belongs to the finite set $\{D\in\mathfs P:D\cap M_t\neq\emptyset\}$.
\item Finiteness of $m$: we have $e^{-m}>Q_\ve(\wh x)>t$, hence $m<|\log t|$.
\item Finiteness of $j$: $t<p\leq e^{-j+2}$, hence $j\leq |\log t|+2$.
\item Finiteness of $p$: $\#\{p\in I_\ve:p>t\}\leq \#(I_\ve\cap (t,1])$ is finite.
\end{enumerate}
The first five items above give that, for $\un a\in\N_0^3$ and $t>0$,
\begin{align*}
\#\left\{\Gamma(\wh x):
\begin{array}{c}
\Psi_{\wh x}^p\in\mathfs A\text{ s.t. }p>t\text{ and }\\
\vt_i[\wh x]\in D_{a_i}\text{ for }|i|\leq 1
\end{array}
\right\}\leq\sum_{j=0}^{\lceil |\log t|\rceil+2}
\sum_{k_i,\ell_i,m=0}^{T_t}
\# Y_{\un k,\un\ell,\un a,m}(j)
\end{align*}
is the finite sum of finite terms, hence finite. Together with the last item above and the
choice of $\mathfs P$, we obtain that
\begin{align*}
\#\left\{\Psi_{\wh x}^p\in\mathfs A:p>t\right\}&\leq 
\sum_{j=0}^{\lceil |\log t|\rceil+2}
\sum_{k_i,\ell_i,m=0}^{T_t}
\# Y_{\un k,\un\ell,\un a,m}(j)\\
&\ \ \ \ \times (\#\{D\in\mathfs P:D\cap M_t\neq\emptyset\})^3\times (\#(I_\ve\cap (t,1]))
\end{align*}
is finite. This proves the discreteness property of $\mathfs A$.

\medskip
\noindent
{\em Proof of sufficiency.}
Let $\wh x\in {\rm NUE}_\chi^*$. Take $(k_i)_{i\in\Z},(\ell_i)_{i\in\Z},(m_i)_{i\in\Z},(a_i)_{i\in\Z},(j_i)_{i\in\Z}$  s.t.:
\begin{align*}
& d(\vt_i[\wh x],\mathfs S)\in [e^{-k_i-1},e^{-k_i}),
\ u(\wh f^i(\wh x))\in [e^{\ell_i},e^{\ell_i+1}),\ Q_\ve(\wh f^i(\wh x))\in [e^{-m_i-1},e^{-m_i}),\\
&\vt_i[\wh x]\in D_{a_i},\ q_\ve(\wh f^i(\wh x))\in[e^{-j_i-1},e^{-j_i+1}).
\end{align*}
For $n\in\Z$, define
\begin{align*}
&\un k^{(n)}=(k_{n-1},k_n,k_{n+1}),\ \un\ell^{(n)}=(\ell_{n-1},\ell_n,\ell_{n+1}),\ \un a^{(n)}=(a_{n-1},a_n,a_{n+1}).
\end{align*}
Then $\Gamma(\wh f^n(\wh x))\in Y_{\un k^{(n)},\un\ell^{(n)},\un a^{(n)},m_n}$.
Take $\Gamma(\wh x_n)\in Y_{\un k^{(n)},\un\ell^{(n)},\un a^{(n)},m_n}(j_n)$ s.t.:
\begin{enumerate}[aaa)]
\item[(${\rm a}_n$)] $ d(\vt_i[\wh f^n(\wh x)],\vt_i[\wh x_n])+
|u(\wh f^{i}(\wh f^n(\wh x)))^{-1}-u(\wh f^i(\wh x_n))^{-1}|<e^{-8(j_n+2)}$ for $|i|\leq 1$.
\item[(${\rm b}_n$)] $\tfrac{Q_\ve(\wh f^n(\wh x))}{Q_\ve(\wh x_n)}=e^{\pm\frac{\ve}{3}}$.
\end{enumerate}
Define $p_n=\delta_\ve\min\{e^{\ve|k|}Q_\ve(\wh x_{n+k}):k\leq 0\}$.
We claim that $\{\Psi_{\wh x_n}^{p_n}\}_{n\in\Z}$ is an $\ve$--gpo
in $\mathfs A^\Z$ that shadows $\wh x$.

\medskip
\noindent
{\sc Claim 2:} $\Psi_{\wh x_n}^{p_n}\in\mathfs A$ for all $n\in\Z$.

\medskip
\noindent
(CG1) By definition, $\Gamma(\wh x_n)\in Y_{\un k^{(n)},\un\ell^{(n)},\un a^{(n)},m_n}(j_n)$.

\noindent
(CG2) By (${\rm b}_n$),
$\min\{e^{\ve|k|}Q_\ve(\wh x_{n+k}):k\leq 0\}=
e^{\pm\frac{\ve}{3}}\min\{e^{\ve|k|}Q_\ve(\wh f^{n+k}(\wh x)):k\leq 0\}$
hence $p_n$ is well-defined and satisfies $p_n=e^{\pm\frac{\ve}{3}}q_\ve(\wh f^n(\wh x))$.
Therefore $p_n\in I_\ve$, $p_n\leq \delta_\ve Q_\ve(\wh x_n)$,
and $p_n\in[e^{-j_n-2},e^{-j_n+2})$.

\medskip
\noindent
{\sc Claim 3:} $\Psi_{\wh x_n}^{p_n}\overset{\ve}{\leftarrow}\Psi_{\wh x_{n+1}}^{p_{n+1}}$
for all $n\in\Z$.

\medskip
\noindent
(E1) By (${\rm a}_{n+1}$) with $i=-1$ and (${\rm a}_n$) with $i=0$,
\begin{align*}
&\ d(\vt[\wh f^{-1}(\wh x_{n+1})],\vt[\wh x_n])+
|u(\wh f^{-1}(\wh x_{n+1}))^{-1}-u(\wh x_n)^{-1}|\\
&\leq d(\vt_{-1}[\wh x_{n+1}],\vt_{-1}[\wh f^{n+1}(\wh x)])+
|u(\wh f^{-1}(\wh x_{n+1}))^{-1}-u(\wh f^n(\wh x))^{-1}|\\
&\ \ \ \,+ d(\vt_0[\wh f^n(\wh x)],\vt_0[\wh x_n])+
|u(\wh f^n(\wh x))^{-1}-u(\wh x_n)^{-1}|\\
&<e^{-8(j_{n+1}+2)}+e^{-8(j_n+2)}.
\end{align*}
Note that 
\begin{align*}
&\ e^{-8(j_{n+1}+2)}+e^{-8(j_n+2)}\leq e^{-8}\left(q_\ve(\wh f^{n+1}(\wh x))^8+
q_\ve(\wh f^n(\wh x))^8\right)\\
&\overset{!}{\leq} e^{-8}(1+e^{8\ve})q_\ve(\wh f^n(\wh x))^8\leq
e^{-8+\frac{8\ve}{3}}(1+e^{8\ve})p_n^8\overset{!!}{<}p_n^8,
\end{align*}
where in $\overset{!}{\leq}$ we used Lemma \ref{Lemma-q}(2) and in $\overset{!!}{<}$ we used
that $e^{-8+\frac{8\ve}{3}}(1+e^{8\ve})<1$ when $\ve>0$ is sufficiently small. Therefore
$\Psi_{\wh f^{-1}(\wh x_{n+1})}^{p_n}\overset{\ve}{\approx}\Psi_{\wh x_n}^{p_n}$.

\medskip
\noindent
(E2) We will use the inequality $e^{-8(j_{n+1}+2)}+e^{-8(j_n+2)}<p_n^8$ proved above.

\medskip
\noindent
(E2.1) As remarked before, it follows directly from condition (E1).


\medskip
\noindent
(E2.2) By (${\rm a}_{n}$) with $i=1$ and (${\rm a}_{n+1}$) with $i=0$ we have
$|u(\wh f^{n+1}(\wh x))^{-1}-u(\wh f(\wh x_n))^{-1}|<e^{-8(j_n+2)}<p_n^8$
and $|u(\wh f^{n+1}(\wh x))^{-1}-u(\wh x_{n+1})^{-1}|<e^{-8(j_{n+1}+2)}<p_{n+1}^8$.
Proceeding as in the proof of Proposition \ref{Lemma-overlap}(1),
this first inequality implies that $\tfrac{u(\wh f^{n+1}(\wh x))}{u(\wh f(\wh x_n))}=e^{\pm p_n^6}$
and the second implies that $\tfrac{u(\wh f^{n+1}(\wh x))}{u(\wh x_{n+1})}=e^{\pm p_{n+1}^6}$.
Since $p_n^6\ll\tfrac{p_n}{2}$ and $p_{n+1}^6\leq e^{6\ve} p_n^6\ll\tfrac{p_n}{2}$, it follows that
$\tfrac{u(\wh f(\wh x_n))}{u(\wh x_{n+1})}=e^{\pm p_n}$.

\medskip
\noindent
(E2.3) The definition of $p_n$ guarantees that $p_{n+1}=\min\{e^\ve p_n,\delta_\ve Q_\ve(\wh x_{n+1})\}$.


\medskip
\noindent
{\sc Claim 4:} $\{\Psi_{\wh x_n}^{p_n}\}_{n\in\Z}$ shadows $\wh x$.

\medskip
By (${\rm a}_n$) with $i=0$, we have
$\Psi_{\wh f^n(\wh x)}^{p_n}\overset{\ve}{\approx}\Psi_{\wh x_n}^{p_n}$, hence 
by Proposition \ref{Lemma-overlap}(2) we have
$\vt_n[\wh x]=\vt[\wh f^n(\wh x)]=\Psi_{\wh f^n(\wh x)}(0)\in \Psi_{\wh x_n}(R[p_n])$,
therefore $\{\Psi_{\wh x_n}^{p_n}\}_{n\in\Z}$ shadows $\wh x$.

\medskip
This concludes the proof of sufficiency. Note that if $\wh x\in{\rm NUE}_\chi^\#$ then the
$\ve$--gpo constructed above belongs to $\Sigma^\#$. This observation will be used in the proof
of Proposition \ref{Prop-pi} below.

\medskip
\noindent
{\em Proof of relevance.} The family $\mathfs A$ might not a priori satisfy
the relevance condition, but we can easily reduce it to a sub-alphabet $\mathfs A'$ satisfying (1)--(3) as follows.
Call $v\in\mathfs A$ relevant if there is $\un v=\{v_n\}_{n\in\Z}\in\mathfs A^\Z$ with $v_0=v$ s.t. $\un{v}$ shadows
a point in ${\rm NUE}_\chi^*$. When this happens then every $v_n$ is relevant
(since ${\rm NUE}_\chi^*$ is $\wh f$--invariant).
Therefore $\mathfs A'=\{v\in\mathfs A:v\text{ is relevant}\}$ is discrete
because $\mathfs A'\subset\mathfs A$, and it is sufficient and relevant by definition.
\end{proof}

Let $\Sigma$ be the TMS associated to the graph with vertex set $\mathfs A$ given by
Theorem \ref{Thm-coarse-graining} and
edges $v\overset{\ve}{\leftarrow}w$. An element of $\Sigma$ is an $\ve$--gpo, hence
we define $\pi:\Sigma\to \wh M$ by
$$
\{\pi[\{v_n\}_{n\in\Z}]\}:=V^s[\{v_n\}_{n\geq 0}]\cap V^u[\{v_n\}_{n\leq 0}].
$$
Here are the main properties of the triple $(\Sigma,\sigma,\pi)$.

\begin{proposition}\label{Prop-pi}
The following holds for all $\ve>0$ small enough.
\begin{enumerate}[{\rm (1)}]
\item $\pi:\Sigma\to {\wh M}$ is H\"older continuous.
\item $\pi\circ\sigma=\wh f\circ\pi$.
\item $\pi[\Sigma^\#]\supset{\rm NUE}_\chi^\#$.
\end{enumerate} 
\end{proposition}

Part (1) follows from Proposition \ref{Prop-stable-manifolds}(3),
part (2) follows from Proposition \ref{Prop-stable-manifolds}(2),
and part (3) is a direct consequence of the observation in the end
of the proof of Theorem \ref{Thm-coarse-graining}(2).

\begin{remark}\label{Remark-finite-degree}
It is important noticing the difference between our $(\Sigma,\sigma)$ and
those constructed in \cite{Sarig-JAMS,Lima-Sarig,Lima-Matheus}:
while the later ones have finite ingoing and outgoing degrees (thus $\Sigma$ is locally compact),
our symbolic space does not necessarily satisfy this. The reason is that condition (E2.3) does not imply
a lower bound on $p_{n+1}$. This non-finiteness property also holds in Hofbauer towers.
\end{remark}

In general $(\Sigma,\sigma,\pi)$ does {\em not} satisfy Theorem \ref{Thm-Main},
since $\pi$ might be infinite-to-one. We use $\pi$ to induce a locally
finite cover of ${\rm NUE}_\chi^\#$, which will then be refined to a partition of ${\rm NUE}_\chi^\#$
that will lead to the proof of Theorem \ref{Thm-Main}.

\section{The inverse problem}

Our goal now is to analyze when $\pi$ loses injectivity. More specifically, given that
$\pi(\un{v})=\pi(\un{w})$ we want to compare $v_n$ with $w_n$ and show that one is defined
by the other ``up to bounded error''. We do this under the additional assumption
that $\un{v},\un{w}\in\Sigma^\#$. Remind that $\Sigma^\#$ is the {\em recurrent set} of $\Sigma$:
$$
\Sigma^\#:=\left\{\un v\in\Sigma:\exists v,w\in V\text{ s.t. }\begin{array}{l}v_n=v\text{ for infinitely many }n>0\\
v_n=w\text{ for infinitely many }n<0
\end{array}\right\}.
$$
The main result of this section is the following theorem.

\begin{theorem}[Inverse theorem]\label{Thm-inverse}
The following holds for $\ve>0$ small enough.
If $\{\Psi_{\wh x_n}^{p_n}\}_{n\in\Z},\{\Psi_{\wh y_n}^{q_n}\}_{n\in\Z}\in\Sigma^\#$ satisfy
$\pi[\{\Psi_{\wh x_n}^{p_n}\}_{n\in\Z}]=\pi[\{\Psi_{\wh y_n}^{q_n}\}_{n\in\Z}]$ then for all $n\in\Z$:
\begin{enumerate}[{\rm (1)}]
\item $d(\vt[\wh x_n],\vt[\wh y_n])\leq 2\max\{p_n,q_n\}$.
\item $\tfrac{u(\wh x_n)}{u(\wh y_n)}=e^{\pm 2\sqrt{\ve}}$.
\item $\tfrac{Q_\ve(\wh x_n)}{Q_\ve(\wh y_n)}=e^{\pm \sqrt[3]{\ve}}$.
\item $\tfrac{p_n}{q_n}=e^{\pm\sqrt[3]{\ve}}$.
\item $(\Psi_{\wh y_n}^{-1}\circ\Psi_{\wh x_n})(t)=t+\Delta_n(t)+\delta_n$
for $t\in R[10Q_\ve(\wh x_n)]$, where $\delta_n\in\R$ with $|\delta_n|<3q_n$ and
$\Delta_n:R[10Q_\ve(\wh x_n)]\to\R$ with $\Delta_n(0)=0$ and
$\|d\Delta_n\|_0<4\sqrt{\ve}$.
\end{enumerate}
\end{theorem}

The substantial differences of the above theorem from \cite[Thm 5.2]{Sarig-JAMS} rely on part (5):
in our case we can only obtain estimates inside the smaller rectangle $R[10Q_\ve(\wh x_n)]$,
and our estimate on $\delta_n$ is slightly weaker. This latter fact is the
reason we introduced weak edges, as it will be clear in section \ref{Section-symbolic-dynamics}
(see Remark \ref{Remark-explanation-weak-gpo}).
Part (1) is proved similarly to \cite[Prop. 5.3]{Sarig-JAMS}, as follows.
Write $\pi[\{\Psi_{\wh x_n}^{p_n}\}_{n\in\Z}]=\pi[\{\Psi_{\wh y_n}^{q_n}\}_{n\in\Z}]=\wh x$.
For each $n\in\Z$, $\vt_n[\wh x]=\Psi_{\wh x_n}(t)$ for some $t\in R[p_n]$.
Since $\Psi_{\wh x_n}(0)=\vt[\wh x_n]$ and $\Psi_{\wh x_n}$ is $1$--Lipschitz,
$d(\vt_n[\wh x],\vt[\wh x_n])\leq p_n$.
Similarly $d(\vt_n[\wh x],\vt[\wh y_n])\leq q_n$,
and so $d(\vt[\wh x_n],\vt[\wh y_n])\leq p_n+q_n$, which is better than stated in part (1).

\medskip
Before proceeding to the other parts, we discuss two consequences of part (1).
The first one is that for $\ve>0$ small enough it holds:
\begin{align}\label{estimates-distances}
\frac{1}{2}\leq \frac{d(\vt_i[\wh x_n],\mathfs S)^a}{d(\vt_i[\wh y_n],\mathfs S)^a}\leq 2,\ \forall |i|\leq 1.
\end{align}
Take $i=0$, and start noting that
$d(\vt[\wh x_n],\vt[\wh y_n])\leq p_n+q_n<\ve[d(\vt[\wh x_n],\mathfs S)+d(\vt[\wh y_n],\mathfs S)]$
hence $d(\vt[\wh x_n],\mathfs S)=d(\vt[\wh y_n],\mathfs S)\pm
d(\vt[\wh x_n],\vt[\wh y_n])
=d(\vt[\wh y_n],\mathfs S)\pm
\ve[d(\vt[\wh x_n],\mathfs S)+d(\vt[\wh y_n],\mathfs S)]$, thus
$\tfrac{1-\ve}{1+\ve}\leq \tfrac{d(\vt[\wh x_n],\mathfs S)}{d(\vt[\wh y_n],\mathfs S)}\leq \tfrac{1+\ve}{1-\ve}$.
If $\ve>0$ is small enough then $\tfrac{1+\ve}{1-\ve}<e^{3\ve}$, and so
\begin{equation}\label{equation-distances-1}
\frac{d(\vt[\wh x_n],\mathfs S)}{d(\vt[\wh y_n],\mathfs S)}=e^{\pm 3\ve}.
\end{equation}
Now take $i=-1$.
By (E1), $d(\vt_{-1}[\wh x_n],\vt[\wh x_{n-1}])\leq p_{n-1}^8<\ve d(\vt[\wh x_{n-1}],\mathfs S)$
and so $d(\vt_{-1}[\wh x_n],\mathfs S)=d(\vt[\wh x_{n-1}],\mathfs S)\pm d(\vt_{-1}[\wh x_n],\vt[\wh x_{n-1}])=
(1\pm\ve)d(\vt[\wh x_{n-1}],\mathfs S)$. If $\ve>0$ is small enough then
$e^{-2\ve}<1-\ve<1+\ve<e^{2\ve}$, therefore $\tfrac{d(\vt_{-1}[\wh x_n],\mathfs S)}{d(\vt[\wh x_{n-1}],\mathfs S)}=e^{\pm 2\ve}$.
Similarly $\tfrac{d(\vt_{-1}[\wh y_n],\mathfs S)}{d(\vt[\wh y_{n-1}],\mathfs S)}=e^{\pm 2\ve}$,
hence (\ref{equation-distances-1}) above implies that
\begin{equation}\label{equation-distances-2}
\frac{d(\vt_{-1}[\wh x_n],\mathfs S)}{d(\vt_{-1}[\wh y_n],\mathfs S)}=e^{\pm 7\ve}.
\end{equation}
It remains to take $i=1$. By (E2.1), 
$d(\vt_1[\wh x_n],\vt[\wh x_{n+1}])<p_n<\ve d(\vt_1[\wh x_n],\mathfs S)$
and so $d(\vt[\wh x_{n+1}],\mathfs S)=(1\pm\ve)d(\vt_1[\wh x_n],\mathfs S)$.
Now proceed as in the case $i=-1$ to get
\begin{equation}\label{equation-distances-3}
\frac{d(\vt_1[\wh x_n],\mathfs S)}{d(\vt_1[\wh y_n],\mathfs S)}=e^{\pm 7\ve}.
\end{equation}
It is clear that (\ref{estimates-distances}) follows from
(\ref{equation-distances-1}), (\ref{equation-distances-2}), (\ref{equation-distances-3}).

\medskip
The second consequence of part (1) is that $\vt[\wh x_n],\vt[\wh y_n]\in D_{\vt[\wh x_n]}\cap D_{\vt[\wh y_n]}$.
To see this, note that (\ref{equation-distances-1}) and (\ref{equation-distances-3}) imply
$\tfrac{1}{2}\leq\tfrac{\min\{d(\vt[\wh x_n],\mathfs S)^a,d(\vt_1[\wh x_n],\mathfs S)^a\}}{\min\{d(\vt[\wh y_n],\mathfs S)^a,d(\vt_1[\wh y_n],\mathfs S)^a\}}\leq 2$ and so
\begin{align*}
&\ d(\vt[\wh x_n],\vt[\wh y_n])<\ve[\min\{d(\vt[\wh x_n],\mathfs S)^a,d(\vt_1[\wh x_n],\mathfs S)^a\}\\
&\hspace{2.7cm}+\min\{d(\vt[\wh y_n],\mathfs S)^a,d(\vt_1[\wh y_n],\mathfs S)^a\}]\\
&<3\ve\min\{d(\vt[\wh y_n],\mathfs S)^a,d(\vt_1[\wh y_n],\mathfs S)^a\}<\mathfrak r(\vt[\wh y_n])
\end{align*}
and thus $\vt[\wh x_n]\in D_{\vt[\wh y_n]}$.
Similarly $\vt[\wh y_n]\in D_{\vt[\wh x_n]}$.
As a consequence, we can apply assumptions (A1)--(A3) with respect to either $\vt[\wh x_n]$
or $\vt[\wh y_n]$.

\subsection{Control of $u(\wh x_n)$}
We now make use of the hyperbolicity of $\wh f$ to show that $u$ improves along
an $\ve$--gpo.

\begin{proposition}
The following holds for all $\ve>0$ small enough.
If $\{\Psi_{\wh x_n}^{p_n}\}_{n\in\Z}$, $\{\Psi_{\wh y_n}^{q_n}\}_{n\in\Z}\in\Sigma^\#$ satisfy
$\pi[\{\Psi_{\wh x_n}^{p_n}\}_{n\in\Z}]=\pi[\{\Psi_{\wh y_n}^{q_n}\}_{n\in\Z}]$ then
$\tfrac{u(\wh x_n)}{u(\wh y_n)}=e^{\pm 2\sqrt{\ve}}$, $\forall n\in\Z$.
\end{proposition}

\begin{proof}
When $M$ is compact and $f$ is a $C^{1+\beta}$ diffeomorphism,
this follows from Lemma 7.2 and Proposition 7.3 of \cite{Sarig-JAMS}. We employ similar methods.
Let $\un v=\{\Psi_{\wh x_n}^{p_n}\}_{n\in\Z}$, $\un w=\{\Psi_{\wh y_n}^{q_n}\}_{n\in\Z}$,
and $\pi[\un v]=\pi[\un w]=\wh x$.

\medskip 
\noindent
{\sc Claim 1 (Improvement lemma):} The following holds for all $\ve>0$ small enough.
For $\xi\geq\sqrt{\ve}$, if $\tfrac{u(\wh f^n(\wh x))}{u(\wh x_n)}=e^{\pm\xi}$
then $\tfrac{u(\wh f^{n+1}(\wh x))}{u(\wh x_{n+1})}=e^{\pm(\xi-Q_\ve(\wh x_{n})^{\beta/4})}$.

\medskip
Note that the ratio improves. 

\begin{proof}[Proof of Claim $1$.]
It is enough to prove the claim for $n=0$, so assume
$\tfrac{u(\wh x)}{u(\wh x_0)}=e^{\pm\xi}$ with $\xi\geq\sqrt{\ve}$.
We have $\tfrac{u(\wh f(\wh x))}{u(\wh x_1)}=
\tfrac{u(\wh f(\wh x))}{u(\wh f(\wh x_0))}\cdot
\tfrac{u(\wh f(\wh x_0))}{u(\wh x_1)}$.
By (E2.2) we have  $\tfrac{u(\wh f(\wh x_0))}{u(\wh x_1)}=e^{\pm p_0}$,
and since $p_0\ll Q_\ve(\wh x_0)^{\beta/4}$ it follows that
$\tfrac{u(\wh f(\wh x_0))}{u(\wh x_1)}=e^{\pm Q_\ve(\wh x_0)^{\beta/4}}$.
Thus it is enough to show that
$\tfrac{u(\wh f(\wh x))}{u(\wh f(\wh x_0))}=e^{\pm(\xi-2Q_\ve(\wh x_0)^{\beta/4})}$.
We show that $\tfrac{u(\wh f(\wh x))}{u(\wh f(\wh x_0))}\leq e^{\xi-2Q_\ve(\wh x_0)^{\beta/4}}$
(the other side is proved similarly). By the proof of Lemma \ref{Lemma-linear-reduction},
\begin{align*}
&\frac{u(\wh f(\wh x))^2}{u(\wh f(\wh x_0))^2}=
\frac{1+e^{2\chi}|df_{\vt[\wh x]}|^{-2}u(\wh x)^2}{1+e^{2\chi}|df_{\vt[\wh x_0]}|^{-2}u(\wh x_0)^2}\leq 
\frac{1+e^{2\xi+2\chi}|df_{\vt[\wh x]}|^{-2}u(\wh x_0)^2}{1+e^{2\chi}|df_{\vt[\wh x_0]}|^{-2}u(\wh x_0)^2}\\
&\leq
\underbrace{\left(\tfrac{1+e^{2\xi+2\chi}|df_{\vt[\wh x_0]}|^{-2}u(\wh x_0)^2}{1+e^{2\chi}|df_{\vt[\wh x_0]}|^{-2}u(\wh x_0)^2}\right)}_{= \text{ I}}\ 
\underbrace{{\rm exp}\left(2\left|\log|df_{\vt[\wh x]}|-\log|df_{\vt[\wh x_0]}|\right|\right)}_{=\text{ II}}.
\end{align*}
Note that I can be written as:
\begin{align*}
\ \text{I}=\tfrac{1+e^{2\xi+2\chi}|df_{\vt[\wh x_0]}|^{-2}u(\wh x_0)^2}{1+e^{2\chi}
|df_{\vt[\wh x_0]}|^{-2}u(\wh x_0)^2}=e^{2\xi}-\tfrac{e^{2\xi}-1}{u(\wh f(\wh x_0))^2}
=e^{2\xi}\left[1-\tfrac{1-e^{-2\xi}}{u(\wh f(\wh x_0))^2}\right].
\end{align*}
Using that $\xi\geq\sqrt{\ve}$ and that $u(\wh f(\wh x_0))<e^{\chi+1}\rho(\wh x_0)^{-a}u(\wh x_0)$
(see the proof of discreteness of Theorem \ref{Thm-coarse-graining}), it follows that for $\ve>0$ small enough it holds:
\begin{align*}
&\tfrac{1-e^{-2\xi}}{u(\wh f(\wh x_0))^2}>\ve^{\frac{1}{2}}e^{-2\chi-2}\rho(\wh x_0)^{2a}u(\wh x_0)^{-2}
>\ve^{\frac{1}{2}}e^{-2\chi-2}\left[\ve^{-\frac{1}{12}}Q_\ve(\wh x_0)^{\frac{\beta}{36}}\right]
\left[\ve^{-\frac{1}{4}}Q_\ve(\wh x_0)^{\frac{\beta}{12}}\right]\\
&=\ve^{\frac{1}{6}}e^{-2\chi-2}Q_\ve(\wh x_0)^{-\frac{5\beta}{36}} Q_\ve(\wh x_0)^{\frac{\beta}{4}}
>\ve^{\frac{1}{6}}e^{-2\chi-2}\ve^{-\frac{5}{12}} Q_\ve(\wh x_0)^{\frac{\beta}{4}}
=\ve^{-\frac{1}{4}}e^{-2\chi-2}Q_\ve(\wh x_0)^{\frac{\beta}{4}}\\
&>5Q_\ve(\wh x_0)^{\frac{\beta}{4}}.
\end{align*}
Using that $1-t<e^{-t}$ for $t\in\R$, we obtain that $\text{I}<e^{2\xi-5Q_\ve(\wh x_0)^{\beta/4}}$.

\medskip
We now estimate II. By (A2)--(A3), if $x,y\in M$ with $y\in D_x$ then
$$
\left|\frac{|df_x|}{|df_y|}-1\right|\leq |df_y|^{-1}|df_x-df_y|
\leq\mathfrak K d(x,\mathfs S)^{-a}|x-y|^\beta.
$$
Since $\log t\leq t-1$ for $t>1$, we get
$|\log|df_x|-\log|df_y||\leq \mathfrak K d(x,\mathfs S)^{-a}|x-y|^\beta$
whenever $y\in D_x$, hence for $\ve>0$ small enough it holds:
\begin{align*}
&\ 2\left|\log|df_{\vt[\wh x]}|-\log|df_{\vt[\wh x_0]}|\right|
\leq 2\mathfrak K d(\vt[\wh x_0],\mathfs S)^{-a}d(\vt[\wh x],\vt[\wh x_0])^\beta\\
&\leq 2\mathfrak K\rho(\wh x_0)^{-a} Q_\ve(\wh x_0)^\beta
<2\mathfrak K \ve^{\frac{1}{24}}Q_\ve(\wh x_0)^{\frac{71\beta}{72}}<Q_\ve(\wh x_0)^{\frac{\beta}{4}}.
\end{align*}
The estimates of I and II above imply that
$\tfrac{u(\wh f(\wh x))}{u(\wh f(\wh x_0))}\leq e^{\xi-2Q_\ve(\wh x_0)^{\beta/4}}$, as claimed.
\end{proof}

\begin{remark}\label{Rmk-after-improvement-lemma}
The proof above also works to show that if $\wh f(\wh x)\in V^u[\{\Psi_{\wh x_n}^{p_n}\}_{n\leq 1}]$
and if $\tfrac{u(\wh x)}{u(\wh x_0)}=e^{\pm\xi}$ for $\xi\geq\sqrt{\ve}$,
then $\tfrac{u(\wh f(\wh x))}{u(\wh x_1)}=e^{\pm(\xi-Q_\ve(\wh x_0)^{\beta/4})}$.
In particular, the ratio does not get worse.
\end{remark}

At this point, we do not yet know that $u(\wh x)<\infty$. The next claim gives this.

\medskip
\noindent
{\sc Claim 2:} $\exists\xi\geq\sqrt{\ve}$ and
a sequence $n_k\to-\infty$ s.t. $\tfrac{u(\wh f^{n_k}(\wh x))}{u(\wh x_{n_k})}=e^{\pm\xi}$ for all $k$.

\medskip
In particular $u(\wh f^{n_k}(\wh x))<\infty$ and so, by the proof of Lemma \ref{Lemma-linear-reduction},
$u(\wh x)<\infty$.

\begin{proof}[Proof of Claim $2$.]
The statement and its proof are very similar to Claim 1 in the proof of \cite[Prop. 7.3]{Sarig-JAMS}.
Since $\un v\in\Sigma^\#$, there is a sequence $n_k\to-\infty$
s.t. $v_{n_k}=v=\Psi_{\wh y}^q$ for all $k$. Thus it is enough to show that $u(\wh f^{n_k}(\wh x))$
is uniformly bounded. Since $v$ is relevant, Theorem \ref{Thm-coarse-graining}(3)
implies that there is $\un w=\{w_n\}_{n\in\Z}\in\Sigma$ with $w_0=v$ and $\pi[\un w]=\wh z\in{\rm NUE}_\chi^*$.
In particular, $u(\wh z)<\infty$. Let $V^u:=V^u[\{w_n\}_{n\leq 0}]$, then by Proposition \ref{Prop-stable-manifolds}(3)(b)
we have $u(\wh z')\leq \exp{}[\ve^{3/4}]u(\wh z)$ for all $\wh z'\in V^u$.
Define $\xi_0>0$ by
$e^{\xi_0}:=\max\left\{\tfrac{u(\wh y)}{u(\wh z)},\tfrac{u(\wh z)}{u(\wh y)},e^{\sqrt{\ve}}\right\}$
and $\xi:=\xi_0+\ve^{3/4}$.
Below we will show that $u(\wh f^{n_k}(\wh x))\leq e^{\xi}u(\wh y)$ for all $k$.

\medskip
Fix $k$. For $\ell<k$ define $\wh z_\ell\in\wh M$ by
$\wh f^{n_\ell}(\wh z_\ell):=V^u\cap V^s[\{v_m\}_{m\geq n_\ell}]$.
This definition makes sense because $v_{n_\ell}=v$ and $V^u$ is an unstable set at $v$.
It is clear that $\wh z_\ell\to\wh x$. Since $\wh z,\wh f^{n_\ell}(\wh z_\ell)\in V^u$, we have that 
$\tfrac{u(\wh f^{n_\ell}(\wh z_\ell))}{u(\wh y)}=\tfrac{u(\wh f^{n_\ell}(\wh z_\ell))}{u(\wh z)}\cdot\tfrac{u(\wh z)}{u(\wh y)}=e^{\pm\xi}$,
hence by Remark \ref{Rmk-after-improvement-lemma} above we have
$\tfrac{u(\wh f^{n_k}(\wh z_\ell))}{u(\wh y)}=e^{\pm\xi}$. For each fixed $N>0$
we have $\wh{df}_{\wh f^{n_k}(\wh z_\ell)}^{(-n)}\xrightarrow[\ell\to\infty]{} \wh{df}_{\wh f^{n_k}(\wh x)}^{(-n)}$
for all $n=0,\ldots,N$, therefore
$$
\sum_{n=0}^N e^{2n\chi}|\wh{df}_{\wh f^{n_k}(\wh x)}^{(-n)}|^2\leq
\lim_{\ell\to\infty}\sum_{n=0}^N e^{2n\chi}|\wh{df}_{\wh f^{n_k}(\wh z_\ell)}^{(-n)}|^2\leq
\lim_{\ell\to\infty} u(\wh f^{n_k}(\wh z_\ell))^2\leq e^{2\xi} u(\wh y)^2
$$
and so $u(\wh f^{n_k}(\wh x))\leq e^{\xi}u(\wh y)$. This completes the proof of the claim.
\end{proof}

Now we complete the proof of the proposition. By Claim 2, we can apply Claim 1
along $\un v$ and the orbit of $\wh x$ for some $\xi\geq\sqrt{\ve}$: if $v_m=v$ for infinitely many $m<n$,
then the ratio improves a definite amount at each of these indices.
The conclusion is that $\tfrac{u(\wh f^n(\wh x))}{u(\wh x_n)}=e^{\pm\sqrt{\ve}}$, and
by the same reason $\tfrac{u(\wh f^n(\wh x))}{u(\wh y_n)}=e^{\pm\sqrt{\ve}}$.
Therefore $\tfrac{u(\wh x_n)}{u(\wh y_n)}=e^{\pm2\sqrt{\ve}}$.
\end{proof}

\subsection{Control of $Q_\ve(\wh x_n)$ and $p_n$}

Recall that $Q_\ve(\wh x):=\max\{q\in I_\ve:q\leq \widetilde Q_\ve(\wh x)\}$ where
$$
\widetilde Q_\ve(\wh x)=\ve^{3/\beta}
\min\left\{u(\wh x)^{-24/\beta},u(\wh f^{-1}(\wh x))^{-12/\beta}\rho(\wh x)^{72a/\beta}\right\},
$$
so we first control $\widetilde Q_\ve(\wh x_n)$.
By part (2), $\tfrac{u(\wh x_n)}{u(\wh y_n)}=e^{\pm2\sqrt{\ve}}$ for all $n\in\Z$.
Since $\Psi_{\wh f^{-1}(\wh x_n)}^{p_{n-1}}\overset{\ve}{\approx}\Psi_{\wh x_{n-1}}^{p_{n-1}}$,
Proposition \ref{Lemma-overlap}(1) implies that
$\tfrac{u(\wh f^{-1}(\wh x_n))}{u(\wh x_{n-1})}=e^{\pm\sqrt{\ve}}$,
and similarly $\tfrac{u(\wh f^{-1}(\wh y_n))}{u(\wh y_{n-1})}=e^{\pm\sqrt{\ve}}$.
Hence
$$
\tfrac{u(\wh f^{-1}(\wh x_n))}{u(\wh f^{-1}(\wh y_n))}=
\tfrac{u(\wh f^{-1}(\wh x_n))}{u(\wh x_{n-1})}\cdot \tfrac{u(\wh x_{n-1})}{u(\wh y_{n-1})}
\cdot \tfrac{u(\wh y_{n-1})}{u(\wh f^{-1}(\wh y_n))}=e^{\pm4\sqrt{\ve}}.
$$

\medskip
Now we estimate the ratio $\tfrac{\rho(\wh x_n)}{\rho(\wh y_n)}$.
By (\ref{equation-distances-1}), (\ref{equation-distances-2}), (\ref{equation-distances-3})
we have $\tfrac{\rho(\wh x_n)}{\rho(\wh y_n)}=e^{\pm 7\ve}$, hence
$\tfrac{\rho(\wh x_n)^{72a}}{\rho(\wh y_n)^{72a}}=e^{\pm\sqrt{\ve}}$.
The conclusion is that
$\tfrac{\widetilde Q_\ve(\wh x_n)}{\widetilde Q_\ve(\wh y_n)}={\rm exp}[\pm(\tfrac{49}{\beta}\sqrt{\ve})]$,
thus $\tfrac{Q_\ve(\wh x_n)}{Q_\ve(\wh y_n)}={\rm exp}[\pm(\frac{2}{3}\ve+\tfrac{49}{\beta}\sqrt{\ve})]$.
For $\ve>0$ small enough we get $\tfrac{Q_\ve(\wh x_n)}{Q_\ve(\wh y_n)}=e^{\pm\sqrt[3]{\ve}}$.

\medskip
We now prove part (4). We use the lemma below, whose proof is the same as in \cite[Prop. 8.3]{Sarig-JAMS}.

\begin{lemma}\label{Lemma-maximality}
If $\{\Psi_{\wh x_n}^{p_n}\}_{n\in\Z}\in\Sigma^\#$ then 
$p_n=\delta_\ve Q_\ve(\wh x_n)$ for infinitely many $n<0$.
\end{lemma}

By symmetry, it is enough to prove that $p_n\geq e^{-\sqrt[3]{\ve}}q_n$ for all $n\in\Z$. We have:
\begin{enumerate}[$\circ$]
\item If $p_n=\delta_\ve Q_\ve(\wh x_n)$ then part (3) gives
$p_n=\delta_\ve Q_\ve(\wh x_n)\geq e^{-\sqrt[3]{\ve}}\delta_\ve Q_\ve(\wh y_n)
\geq e^{-\sqrt[3]{\ve}}q_n$.
\item If $p_n\geq e^{-\sqrt[3]{\ve}}q_n$ then (E2.3) and part (3) imply that
$$p_{n+1}=\min\{e^\ve p_n,\delta_\ve Q_\ve(\wh x_{n+1})\}\geq
e^{-\sqrt[3]{\ve}}\min\{e^\ve q_n,\delta_\ve Q_\ve(\wh y_{n+1})\}=e^{-\sqrt[3]{\ve}}q_{n+1}.$$
\end{enumerate}
By Lemma \ref{Lemma-maximality}, we conclude that $p_n\geq e^{-\sqrt[3]{\ve}}q_n$ for all $n\in\Z$.

\subsection{Control of $\Psi_{\wh y_n}^{-1}\circ\Psi_{\wh x_n}$}

We can assume that $n=0$ and write $\Psi_{\wh x_0}^{p_0}=\Psi_{\wh x}^p$,
$\Psi_{\wh y_0}^{q_0}=\Psi_{\wh y}^q$. Firstly note that, since $1-2t<e^{-t}<e^t<1+2t$
for small $t>0$, part (2) implies $\left|\tfrac{u(\wh y)}{u(\wh x)}-1\right|< 4\sqrt{\ve}$.
%
%
As in the proof of Proposition \ref{Lemma-overlap}(3), we have
$(\Psi_{\wh y}^{-1}\circ \Psi_{\wh x}-{\rm Id})(t)=\Delta(t)+\delta$, where
$\Delta(t):=\left(\tfrac{u(\wh y)}{u(\wh x)}-1\right)t$. Thus:
\begin{enumerate}[$\circ$]
\item $\Delta(0)=0$ and $\|d\Delta\|_0=\left|\tfrac{u(\wh y)}{u(\wh x)}-1\right|<4\sqrt{\ve}$.
\item There are $t\in R[p],t'\in R[q]$ s.t.
$\Psi_{\wh x}(t)=\Psi_{\wh y}(t')=\vt[\wh x]$, hence
$t'=\Psi_{\wh y}^{-1}\circ \Psi_{\wh x}(t)=t+\Delta(t)+\delta$. If $\ve>0$ is small enough, part (4) implies:
\begin{align*}
&|\delta|\leq |t|+\|d\Delta\|_0|t|+|t'|\leq (1+4\sqrt{\ve})p+q\leq [(1+4\sqrt{\ve})e^{\sqrt[3]{\ve}}+1]q<3q.
\end{align*}
\end{enumerate}
This completes the proof of part (5), and hence of Theorem \ref{Thm-inverse}.

\section{Symbolic dynamics}\label{Section-symbolic-dynamics}

\subsection{A countable Markov partition}

Remind that $(\Sigma,\sigma)$ is the TMS constructed from Theorem \ref{Thm-coarse-graining},
and $\pi:\Sigma\to {\wh M}$ is the map defined in the end of section \ref{Section-coarse-graining}.
We now employ Theorem \ref{Thm-inverse} to build a cover of ${\rm NUE}_\chi^\#$
that is locally finite and satisfies a (symbolic) Markov property. We will use the
constructions of \cite{Sarig-JAMS,Lima-Sarig,Lima-Matheus} to build a Markov partition for $\wh f$,
paying attention to the following facts:
\begin{enumerate}[$\circ$]
\item Our stable and unstable sets are not curves, but they do have good descriptions in
terms of the coordinates of $\wh x$ (Lemma \ref{Lemma-stable-sets}), and they do satisfy
properties analogous to their smooth versions (Proposition \ref{Prop-stable-manifolds}).
\item Most of the methods of \cite{Sarig-JAMS,Lima-Sarig,Lima-Matheus} used to
construct the Markov partition are abstract and rely on the properties of stable and unstable sets
that will be stated in this section.
\end{enumerate}

\medskip
\noindent
{\sc The Markov cover $\mathfs Z$:} Let $\mathfs Z:=\{Z(v):v\in\mathfs A\}$, where
$$
Z(v):=\{\pi(\un v):\un v\in\Sigma^\#\text{ and }v_0=v\}.
$$

\medskip
In other words, we take the natural partition of $\Sigma^\#$ into
cylinders at the zeroth position and use $\pi$ to induce a cover $\mathfs Z$.
Stable/unstable sets allow us to define ``invariant fibres'' inside each $Z\in\mathfs Z$, as follows.
Let $Z=Z(v)$.

%

\medskip
\noindent
{\sc $s$/$u$--fibres in $\mathfs Z$:} Given $\wh x\in Z$, let
$W^s(\wh x,Z):=V^s[\{\Psi_{\wh x_n}^{100p_n}\}_{n\geq 0}]\cap Z$
be the {\em $s$--fibre} of $\wh x$ in $Z$ for some (any) $\un v=\{\Psi_{\wh x_n}^{p_n}\}_{n\in\Z}\in\Sigma^\#$
s.t. $\pi(\un v)=\wh x$ and $\Psi_{\wh x_0}^{p_0}=v$.
Similarly, let $W^u(\wh x,Z):=V^u[\{\Psi_{\wh x_n}^{100p_n}\}_{n\leq 0}]\cap Z$
be the {\em $u$--fibre} of $\wh x$ in $Z$.

\medskip
We also define $V^s(\wh x,Z):=V^s[\{\Psi_{\wh x_n}^{100p_n}\}_{n\geq 0}]$
and $V^u(\wh x,Z):=V^u[\{\Psi_{\wh x_n}^{100p_n}\}_{n\leq 0}]$. These sets
are well-defined because $\{\Psi_{\wh x_n}^{100p_n}\}_{n\in\Z}$ is a weak $\ve$--gpo.
By Proposition \ref{Prop-stable-manifolds}(4), the definitions of $V^{s/u}(\wh x,Z),W^{s/u}(\wh x,Z)$
do not depend on the choice of $\un v$, and any two $s$--fibres ($u$--fibres) either coincide or are disjoint.
It is important noticing the relation between $W^{s/u}(\wh x,Z)$ and $V^{s/u}(\wh x,Z)$:
\begin{enumerate}[$\circ$]
\item $W^s(\wh x,Z)=V^s(\wh x,Z)=\vt^{-1}[x]$, where $x\in M$ is uniquely defined by
$f^n(x)\in \Psi_{\wh x_n}(R[p_n])$ for all $n\geq 0$, see Lemma \ref{Lemma-stable-sets}(1).
\item $V^u(\wh x,Z)$ is isomorphic to the interval $\Psi_{\wh x_0}(R[100p_0])$, while
$W^u(\wh x,Z)$ is isomorphic to a (usually fractal) subset of $\Psi_{\wh x_0}(R[p_0])$.
\end{enumerate}

\begin{remark}\label{Remark-explanation-weak-gpo}
In \cite{Sarig-JAMS,Lima-Sarig,Lima-Matheus}, $Z(\Psi_{\wh x}^p)\subset \Psi_{\wh x}(R[10^{-2}p])$
and the estimate for $\delta$ in Theorem \ref{Thm-inverse}(5) is $|\delta|<10^{-1}q$. In particular,
defining $V^u(\wh x,Z):=V^u[\{\Psi_{\wh x_n}^{p_n}\}_{n\leq 0}]$ (without taking
larger domains) is enough to guarantee that Smale brackets of points in $Z,Z'$ are well-defined
whenever $Z\cap Z'\neq\emptyset$. In our case, $Z(\Psi_{\wh x}^p)\subset \Psi_{\wh x}(R[p])$
and the estimate for $\delta$ on Theorem \ref{Thm-inverse}(5) is $|\delta|<3q$. That is why we define
$V^u(\wh x,Z)$ in a larger domain. This change is merely technical, since the statements and proofs
about $V^{s/u}(\wh x,Z)$ in our case are essentially the same as in \cite{Sarig-JAMS,Lima-Sarig,Lima-Matheus}.
\end{remark}

Below we collect the main properties of $\mathfs Z$.

\begin{proposition}\label{Prop-Z}
The following holds for all $\ve>0$ small enough.
\begin{enumerate}[{\rm (1)}]
\item {\sc Covering property:} $\mathfs Z$ is a cover of ${\rm NUE}_\chi^\#$.
\item {\sc Local finiteness:} For every $Z\in\mathfs Z$, $\#\{Z'\in\mathfs Z:Z\cap Z'\neq\emptyset\}<\infty$.
\item {\sc Product structure:} For every $Z\in\mathfs Z$ and every $\wh x,\wh y\in Z$, the intersection
$W^s(\wh x,Z)\cap W^u(\wh y,Z)$ consists of a single element of $Z$.
\item {\sc Symbolic Markov property:} If $\wh x=\pi(\un v)$ with $\un v=\{v_n\}_{n\in\Z}\in\Sigma^\#$ then
$$
\wh f(W^s(\wh x,Z(v_0)))\subset W^s(\wh f(\wh x),Z(v_1))\, \text{ and }\,
f^{-1}(W^u(\wh f(\wh x),Z(v_1)))\subset W^u(\wh x,Z(v_0)).
$$
\end{enumerate}
\end{proposition}

Part (1) follows from Theorem \ref{Thm-coarse-graining}(2),
part (2) follows from Theorem \ref{Thm-inverse}(4) and Theorem \ref{Thm-coarse-graining}(1),
part (3) follows from Proposition \ref{Prop-stable-manifolds}(1) and the product structure of $\Sigma^\#$,
and part (4) is proved as in \cite[Prop. 10.9]{Sarig-JAMS}.
For $\wh x,\wh y\in Z$, let $[\wh x,\wh y]_Z$ denote the intersection element of $W^s(\wh x,Z)$ and $W^u(\wh y,Z)$.

\begin{lemma}
The following holds for all $\ve>0$ small enough.
\begin{enumerate}[{\rm (1)}]
\item {\sc Compatibility:} If $\wh x,\wh y\in Z(v_0)$ and $\wh f(\wh x),\wh f(\wh y)\in Z(v_1)$ with
$v_0\overset{\ve}{\leftarrow} v_1$ then $\wh f([\wh x,\wh y]_{Z(v_0)})=[\wh f(\wh x),\wh f(\wh y)]_{Z(v_1)}$.
\item {\sc Overlapping charts properties:} If $Z=Z(\Psi_{\wh x}^p),Z'=Z(\Psi_{\wh y}^q)\in\mathfs Z$
with $Z\cap Z'\neq \emptyset$ then:
\begin{enumerate}[{\rm (a)}]
\item $Z\subset \Psi_{\wh y}(R[100q])$.
\item If $\wh x\in Z\cap Z'$ then $W^{s/u}(\wh x,Z)\subset V^{s/u}(\wh x,Z')$. 
\item If $\wh z\in Z,\wh w\in Z'$ then $V^s(\wh z,Z)$ and $V^u(\wh w,Z')$ intersect at a unique element. 
\end{enumerate}
\end{enumerate}
\end{lemma}

\begin{proof}
(1)  Proceed as in \cite[Lemma 10.7]{Sarig-JAMS}.

\medskip
\noindent
(2) By Theorem \ref{Thm-inverse}, we have $\tfrac{p}{q}=e^{\pm\sqrt[3]{\ve}}$
and $(\Psi_{\wh y}^{-1}\circ\Psi_{\wh x})(t)=t+\Delta(t)+\delta$
for $t\in R[10Q_\ve(\wh x)]$, where $\delta\in\R$ with $|\delta|<3q$ and
$\Delta:R[10Q_\ve(\wh x)]\to\R$ with $\Delta(0)=0$ and
$\|d\Delta\|_0<4\sqrt{\ve}$. Let us prove (a). Since $Z\subset \Psi_{\wh x}(R[p])$,
it is enough to show that $\Psi_{\wh x}(R[p])\subset \Psi_{\wh y}(R[100q])$.
Here is the proof: if $x=\Psi_{\wh x}(t)$ with $|t|\leq p$ then $x=\Psi_{\wh y}(t')$ with
$t'=(\Psi_{\wh y}^{-1}\circ\Psi_{\wh x})(t)$, and
$|t'|\leq |t|+|\Delta(t)|+|\delta|\leq (1+4\sqrt{\ve})p+3q\leq [(1+4\sqrt{\ve})e^{\sqrt[3]{\ve}}+3]q<100q$.

\medskip
Now we prove (b). Let $\un v,\un w\in\Sigma^\#$ s.t. $\wh x=\pi(\un v)=\pi(\un w)$, where
$\un v=\{\Psi_{\wh x_n}^{p_n}\}_{n\in\Z}$ and $\un w=\{\Psi_{\wh y_n}^{q_n}\}_{n\in\Z}$
with $\Psi_{\wh x_0}^{p_0}=\Psi_{\wh x}^p$ and $\Psi_{\wh y_0}^{q_0}=\Psi_{\wh y}^q$.
Proceeding as in the last paragraph, Theorem \ref{Thm-inverse} implies that
$\Psi_{\wh x_n}(R[p_n])\subset \Psi_{\wh y_n}(R[100q_n])$ and
$\Psi_{\wh y_n}(R[q_n])\subset \Psi_{\wh x_n}(R[100p_n])$ for all $n\in\Z$. By the discussion before
Lemma \ref{Lemma-temperedness}, we get that $g_{\wh x_n}$ and $g_{\wh y_n}$
are the same inverse branch of $f$, for every $n\in\Z$. Since
\begin{align*}
&W^u(\wh x,Z)\subset\{(x_n)_{n\in\Z}\in\wh M:x_0\in \Psi_{\wh x_0}(R[p_0])\text{ and }
x_{n-1}=g_{\wh x_n}(x_n),\forall n\leq 0\},\\
&V^u(\wh x,Z')=\{(x_n)_{n\in\Z}\in\wh M:x_0\in \Psi_{\wh y_0}(R[100q_0])\text{ and }
x_{n-1}=g_{\wh y_n}(x_n),\forall n\leq 0\}
\end{align*}
and $ \Psi_{\wh x_0}(R[p_0])\subset \Psi_{\wh y_0}(R[100q_0])$, it follows that
$W^u(\wh x,Z)\subset V^u(\wh x,Z')$. The other inclusion is actually an equality:
$W^s(\wh x,Z)=V^s(\wh x,Z)=\vt^{-1}[\vt[\wh x]]=V^s(\wh x,Z')$.

\medskip
Now we prove (c). Write $V^s(\wh z,Z)=\vt^{-1}[z]$ and
$V^u(\wh w,Z')=V^u[\{\Psi_{\wh y_n}^{100q_n}\}_{n\leq 0}]$, where $\Psi_{\wh y_0}^{q_0}=\Psi_{\wh y}^q$.
Define $(x_n)_{n\in\Z}$ by $x_n=f^n(z)$ for $n\geq 0$ and $x_{n-1}=g_{\wh y_n}(x_n)$ for $n\leq 0$.
We have $(x_n)_{n\in\Z}\in V^s(\wh z,Z)$. Since $z\in\Psi_{\wh x}(R[p])\subset \Psi_{\wh y}(R[100q])$,
we also have that $(x_n)_{n\in\Z}\in V^u(\wh w,Z')$, hence $V^s(\wh z,Z)\cap V^u(\wh w,Z')$
contains $(x_n)_{n\in\Z}$. Clearly, any element in this intersection must be equal to $(x_n)_{n\in\Z}$.
\end{proof}


\medskip
The next step is to apply a refinement method to destroy non-trivial intersections in $\mathfs Z$. 
The result is a pairwise disjoint cover of ${\rm NUE}_\chi^\#$ with a (geometrical) Markov property.
This idea, originally developed by Sina{\u\i} and Bowen
for finite covers \cite{Sinai-Construction-of-MP,Sinai-MP-U-diffeomorphisms,Bowen-LNM},
works equally well for locally finite countable covers \cite{Sarig-JAMS}.
Let $\mathfs Z=\{Z_1,Z_2,\ldots\}$.

\medskip
\noindent
{\sc The Markov partition $\mathfs R$:} For every $Z_i,Z_j\in\mathfs Z$, define a partition of $Z_i$ by:
\begin{align*}
T_{ij}^{su}&=\{\wh x\in Z_i: W^s(\wh x,Z_i)\cap Z_j\neq\emptyset,
W^u(\wh x,Z_i)\cap Z_j\neq\emptyset\}\\
T_{ij}^{s\emptyset}&=\{\wh x\in Z_i: W^s(\wh x,Z_i)\cap Z_j\neq\emptyset,
W^u(\wh x,Z_i)\cap Z_j=\emptyset\}\\
T_{ij}^{\emptyset u}&=\{\wh x\in Z_i: W^s(\wh x,Z_i)\cap Z_j=\emptyset,
W^u(\wh x,Z_i)\cap Z_j\neq\emptyset\}\\
T_{ij}^{\emptyset\emptyset}&=\{\wh x\in Z_i: W^s(\wh x,Z_i)\cap Z_j=\emptyset,
W^u(\wh x,Z_i)\cap Z_j=\emptyset\}.
\end{align*}
Let $\mathfs T:=\{T_{ij}^{\alpha\beta}:i,j\geq 1,\alpha\in\{s,\emptyset\},\beta\in\{u,\emptyset\}\}$,
and let $\mathfs R$ be the partition generated by $\mathfs T$.


\medskip
Since $T_{ii}^{su}=Z_i$, $\mathfs R$ is a pairwise disjoint cover of ${\rm NUE}_\chi^\#$.
Clearly, $\mathfs R$ is finer than $\mathfs Z$. The local finiteness of $\mathfs Z$ (Proposition \ref{Prop-Z}(2))
implies two local finiteness properties for $\mathfs R$:
\begin{enumerate}[$\circ$]
\item For every $Z\in\mathfs Z$, $\#\{R\in\mathfs R:R\subset Z\}<\infty$.
\item For every $R\in\mathfs R$, $\#\{Z\in\mathfs Z:Z\supset R\}<\infty$.
\end{enumerate}

\medskip
Now we show that $\mathfs R$ is a Markov partition in the sense of Sina{\u\i} \cite{Sinai-MP-U-diffeomorphisms}. For that, we define $s$/$u$--fibres in $\mathfs R$.

\medskip
\noindent
{\sc $s$/$u$--fibres in $\mathfs R$:} Given $\wh x\in R\in\mathfs R$, we define the {\em $s$--fibre}
and {\em $u$--fibre} of $\wh x$ by:
\begin{align*}
W^s(\wh x,R):=
\bigcap_{T_{ij}^{\alpha\beta}\in\mathfs T\atop{T_{ij}^{\alpha\beta}\supset R}} W^s(\wh x,Z_i)\cap T_{ij}^{\alpha\beta}
\, \text{ and }\, W^u(\wh x,R):=
\bigcap_{T_{ij}^{\alpha\beta}\in\mathfs T\atop{T_{ij}^{\alpha\beta}\supset R}}  W^u(\wh x,Z_i)\cap T_{ij}^{\alpha\beta}.
\end{align*}

It is clear that any two $s$--fibres ($u$--fibres) either coincide or are disjoint.

\begin{proposition}\label{Prop-R}
The following holds for $\ve>0$ small enough.
\begin{enumerate}[{\rm (1)}]
\item {\sc Product structure:} For every $R\in\mathfs R$ and every $\wh x,\wh y\in R$, the intersection
$W^s(\wh x,R)\cap W^u(\wh y,R)$ consists of a single element of $R$. Denote it by $[\wh x,\wh y]$.
\item {\sc Hyperbolicity:} If $\wh z,\wh w\in W^s(\wh x,R)$ then
$d(\wh f^n(\wh z),\wh f^n(\wh w))\xrightarrow[n\to\infty]{}0$, and
if $\wh z,\wh w\in W^u(\wh x,R)$ then $d(\wh f^n(\wh z),\wh f^n(\wh w))\xrightarrow[n\to-\infty]{}0$.
The rates are exponential.
\item {\sc Geometrical Markov property:} Let $R_0,R_1\in\mathfs R$. If $\wh x\in R_0$
and $\wh f(\wh x)\in R_1$ then 
$$
\wh f(W^s(\wh x,R_0))\subset W^s(\wh f(\wh x),R_1)\, \text{ and }\, 
\wh f^{-1}(W^u(\wh f(\wh x),R_1))\subset W^u(\wh x,R_0).
$$
\end{enumerate}
\end{proposition}

When $M$ is compact and $f$ is a diffeomorphism, this is \cite[Prop. 11.5 and 11.7]{Sarig-JAMS}
and the same proof works in our case.

\subsection{A finite-to-one Markov extension}

We construct a new coding for $\wh f$.
Let $\widehat{\mathfs G}=(\widehat V,\widehat E)$ be the oriented graph with vertex set
$\widehat V=\mathfs R$ and edge set $\widehat E=\{R\to S:R,S\in\mathfs R\text{ s.t. }\wh f(R)\cap S\neq\emptyset\}$,
and let $(\widehat\Sigma,\widehat\sigma)$ be the TMS induced by $\widehat{\mathfs G}$.

\medskip
For $\ell\in\Z$ and a path $R_m\to\cdots\to R_n$ on $\widehat{\mathfs G}$ define
$_\ell[R_m,\ldots,R_n]:=\wh f^{-\ell}(R_m)\cap\cdots\cap\wh f^{-\ell-(n-m)}(R_n)$,
which is the set of points whose itinerary between the iterates $\ell,\ldots,\ell+(n-m)$ visits the rectangles
$R_m,\ldots,R_n$.
The crucial property of these sets is that $_\ell[R_m,\ldots,R_n]\neq\emptyset$.
This follows by induction, using the Markov property of $\mathfs R$ (Proposition \ref{Prop-R}(3)).

\medskip
The map $\pi$ defines similar sets: for $\ell\in\Z$ and a path
$v_m\overset{\ve}{\leftarrow}\cdots\overset{\ve}{\leftarrow}v_n$ on $\Sigma$ let
$
Z_\ell[v_m,\ldots,v_n]:=\{\pi(\un w):\un w\in\Sigma^\#\text{ and }w_\ell=v_m,\ldots,w_{\ell+(n-m)}=v_n\}$.
There is a relation between $\Sigma$ and $\widehat\Sigma$ in terms of these sets.

\begin{lemma}
If $\{R_n\}_{n\in\Z}\in\widehat\Sigma$ then there exists $\{v_n\}_{n\in\Z}\in\Sigma$
s.t. $R_n\subset Z(v_n)$ and $_{-n}[R_{-n},\ldots,R_n]\subset Z_{-n}[v_{-n},\ldots,v_n]$
for all $n\geq 0$.
\end{lemma}

\begin{proof}
When $M$ is compact and $f$ is a diffeomorphism, this is \cite[Lemma 12.2]{Sarig-JAMS},
whose proof applies a diagonal argument using the fact that
each vertex of $\Sigma$ has finite ingoing and outgoing degree.
Since our $\Sigma$ does do not necessarily satisfy this finiteness
property (see Remark \ref{Remark-finite-degree}), the same proof does {\em not} work in our case.
Instead, we use the local finiteness of $\mathfs R$ to apply the diagonal argument, as follows.

\medskip
For $n\geq 0$, take $\wh x_n\in{}_{-n}[R_{-n},\ldots,R_n]$, and let
$\un v^{(n)}=\{v^{(n)}_{k}\}_{k\in\Z}\in\Sigma^\#$
s.t. $\pi(\un v^{(n)})=\wh x_n$. As in \cite[Lemma 12.2]{Sarig-JAMS},
$_{-n}[R_{-n},\ldots,R_n]\subset Z_{-n}[v^{(n)}_{-n},\ldots,v^{(n)}_n]$.
Since $R_k$ is included in finitely many elements of $\mathfs Z$,
there are finitely many choices for $(v^{(n)}_{-n},\ldots,v^{(n)}_{n})$. 
By a diagonal argument, there is $\un v$ s.t. for all $n\geq 0$,
$(v_{-n},\ldots,v_{n})=(v^{(m)}_{-n},\ldots,v^{(m)}_n)$ for infinitely many $m$.
Now continue as in \cite[Lemma 12.2]{Sarig-JAMS}.
\end{proof}

By Proposition \ref{Prop-R}(2), $\bigcap_{n\geq 0}\overline{_{-n}[R_{-n},\ldots,R_n]}$
is the intersection of a descending chain of nonempty closed sets with
diameters converging to zero.

\medskip
\noindent
{\sc The map $\widehat\pi:\widehat\Sigma\to \wh M$:} Given $\un R=\{R_n\}_{n\in\Z}\in\widehat\Sigma$,
$\widehat\pi(\un R)$ is defined by the identity
$$
\{\widehat\pi(\un R)\}:=\bigcap_{n\geq 0}\overline{_{-n}[R_{-n},\ldots,R_n]}.
$$

\medskip
The triple $(\widehat\Sigma,\widehat\sigma,\widehat\pi)$ satisfies Theorem \ref{Thm-Main}.

\begin{theorem}\label{Thm-widehat-pi}
The following holds for all $\ve>0$ small enough.
\begin{enumerate}[{\rm (1)}]
\item $\widehat\pi:\widehat\Sigma\to\wh M$ is H\"older continuous.
\item $\widehat\pi\circ\widehat\sigma=\wh f\circ\widehat\pi$.
\item $\widehat\pi[\widehat\Sigma^\#]\supset {\rm NUE}_\chi^\#$.
\item Every point of $\widehat\pi[\widehat\Sigma^\#]$ has finitely many pre-images in $\widehat\Sigma^\#$.
\end{enumerate}
\end{theorem}

In particular, if $\mu$ is an $f$--adapted $\chi$--expanding measure and $\wh\mu$ is
its lift to the natural extension, then $\wh\mu(\wh\pi[\widehat\Sigma^\#])=1$.
When $M$ is compact and $f$ is a diffeomorphism, parts (1)--(3) are \cite[Thm. 12.5]{Sarig-JAMS}
and part (4) is \cite[Thm. 5.6(4)]{Lima-Sarig}.
The same proofs work in our case, and the bound
on the number of pre-images is exactly the same: there is a function 
$N:\mathfs R\to\N$ s.t. if $\wh x=\widehat\pi(\un R)$ with $R_n=R$ for infinitely many $n>0$ and $R_n=S$
for infinitely many $n<0$ then $\#\{\un S\in\widehat\Sigma^\#:\widehat\pi(\un S)=\wh x\}\leq N(R)N(S)$.

\bibliographystyle{alpha}
\bibliography{bibliography}{}

\end{document}